\documentclass[]{article}
\usepackage[margin=2.6cm]{geometry}
\usepackage{graphicx,amsfonts,subfigure}
\usepackage{amsmath,amssymb,setspace}
\usepackage{stmaryrd}
\usepackage{bbm}
\usepackage[mathscr]{euscript}
\usepackage{hyperref}
\usepackage{algorithmic}
\usepackage{leftidx}

\numberwithin{equation}{section}

\newtheorem{thm}{Theorem}[section]
\newtheorem{cor}[thm]{Corollary}

\newtheorem{lemma}[thm]{Lemma}
\newtheorem{Remark}[thm]{Remark}

\newtheorem{remark}[thm]{Remark}

\newtheorem{definition}[thm]{Definition}

\newenvironment{proof}{{\bf Proof.}}{\hfill$\square$\vskip.5cm}

\title{Tensor and Matrix Inversions with Applications}

\author{Michael Brazell\footnote{Department of Mechanical and Aeronautical Engineering, Clarkson University, Potsdam, New York 13699, USA.} \,   Na Li\footnote{Department of Mathematics, Clarkson University, Potsdam, NY 13699, USA.} \, Carmeliza Navasca${}^{\dag}$\footnote{Corresponding author: \{cnavasca@clarkson.edu\}} \, Christino Tamon\footnote{Department of Computer Science, Clarkson University, Potsdam, New York 13699, USA.}}

\date{\today}

\begin{document}

\maketitle

\begin{abstract}
\setcounter{section}{0}
Higher order tensor inversion is possible for even order. We have shown that a tensor group endowed with the Einstein (contracted) product is isomorphic to the general linear group of degree $n$.   With the isomorphic group structures, we derived new tensor decompositions which we have shown to be related to the well-known canonical polyadic decomposition and multilinear SVD. Moreover, within this group structure framework, multilinear systems are derived, specifically, for solving high dimensional PDEs and large discrete quantum models. We also address multilinear systems which do not fit the framework in the least-squares sense, that is, when the tensor has an odd number of  modes or when the tensor has distinct dimensions in each modes. With the notion of tensor inversion, multilinear systems are solvable. Numerically we solve multilinear systems using iterative techniques, namely biconjugate gradient and Jacobi methods in tensor format.

\vspace{8pt}
\vskip .2 cm
\end{abstract}

\section{Introduction}
Tensor decompositions have been succesfully applied across many fields which include among others, chemometrics \cite{Pere}, signal processing \cite{Comon,DeLathauwer} and computer vison \cite{Vasilescu}. More recent applications are in large-scale PDEs through a reduced rank representation of operators with applications to quantum chemistry \cite{Khoromskij} and aerospace engineering \cite{Doostan}. Beylkin and Mohlenkamp \cite{BM1,BM2} used a technique called separated representation to obtain a low rank representation of multidimensional operators in quantum models; see \cite{BM1,BM2}. Hackbusch, Khoromskij and Tyrtyshnikov \cite{K1,K2} have solved multidimensional boundary and eigenvalue problems using a reduced low dimensional tensor-product space through separated representation and hierarchical Kronecker tensor from the underlying high spatial dimensions. See the survey papers \cite{DeLathauwer,Khoromskij,Kolda} and the references therein for more applications and tensor based methods. Extensive studies (e.g. \cite{LekHeng,LekHeng2,HOSVD,Kolda2}) have exposed many aspects of the differences between tensors and matrices despite that tensors are multidimensional generalizations of matrices.

In this paper, we continue to investigate the relationship between matrices and tensors. Here we address the questions: when is it possible to matricize (\emph{tensorize}) and apply matrix (tensor) based methods to high dimensional problems and data with inherent tensor (matrix) structure. Specifically, we address tensor inversion through group theoretic structures and by providing numerical methods for specific multilinear systems in quantum mechanical models and high-dimensional PDEs. Since the inversion of tensor impinges upon a tensor-tensor multiplication definition, the contracted product for tensor multiplication was chosen since it provides a natural setting for multilinear systems and high-dimensional eigenvalue problems considered here.  It is also an intrinsic extension of the matrix product rule.  Still other choices of multiplication rules could be considered as well for particular application in hand. For example, in the matrix case, there are the alternative multiplication of Strassen \cite{Strassen} which improves the computational complexity by using block structure format and the optimized matrix multiplication based on \emph{blocking} for improving cache performance by Demmel \cite{DemmelNotes}. In a recent work of Van Loan \cite{VanLoan}, the idea of blocking are extended to tensors.  Our choice of the standard canonical tensor-tensor multiplication provides a useful setting for algorithms for decompositions, inversions and multilinear iterative solvers.

Like tensors, multilinear systems are ubiquitous since they model many phenomena in engineering and sciences. In the field of continuum physics and engineering, isotropic and anisotropic elastic models \cite{LaiRubinKrempl} are multilinear systems. Multilinear systems are also prevalent in the numerical methods for solving partial differential equations (PDEs) in high dimensions, although most tensor based methods for PDEs require a reduction of the spatial dimensions and some applications of tensor decomposition techniques. Here we focus on the iterative methods for solving the Poisson problems in high dimension in a tensor format. Tensor representations are also common in large discrete quantum models like the discrete Schr\"odinger and Anderson models. The study of spectral theory of the Anderson model is a very active research topic. The Anderson model \cite{Anderson}, Anderson's celebrated and ultimately Nobel prize winning work is the archetype and most studied model for understanding the spectral and transport properties of an electron in a disordered medium. Yet there are still many open problems and conjectures for high dimensional $d \geq 3$ cases; see \cite{Dirk,Kirsch,Stolz} and the references therein. The Hamiltonian of the discrete Schr\"odinger and Anderson models are tensors with an even number of modes; they also satisfy the symmetries required in the tensor SVD we described. Moreover, computing the eigenvectors to check for localization properties not only demonstrate the efficacy of our algorithms, but it actually gives some validation and provide some insights to some of the conjectures \cite{Dirk,Kirsch,Stolz}. Recently, Bai et al. \cite{Bai} have solved some key questions in quantum statistical mechanics numerically. For instance, they have developed numerical linear algebra methods for the many-electrons Hubbard model and quantum Monte Carlo simulations. Numerical (multi)linear algebra techniques are increasingly becoming useful tools in understanding very complicated models and very difficult problems in quantum statistical mechanics.

The contribution of this paper is three-fold. First, we define the tensor group which provides the framework for formulating multilinear systems and tensor inversion. Second, we discuss tensor decompositions derived from the isomorphic group structure and relate them to the standard tensor decompositions, namely, canonical polyadic (CP) \cite{CarolChang,Harshman} and multilinear SVD decompositions \cite{Tucker1,Tucker2,Tucker3,HOSVD}. We have shown that the tensor decompositions from the isomorphic properties are special cases of the well-known CP and multilinear SVD with \emph{symmetries} while satisfying some conditions. Stegeman \cite{Stegeman1,Stegeman2} extended Kruskal's existence and uniqueness conditions for CP decomposition for cases with various forms of symmetries (i.e. existence of identical factors). These decompositions appear in many signal processing applications; e.g. see \cite{Comon} and the references therein. When the tensor has the same dimension in all modes, the tensor eigenvalue decomposition in Section 3 is the tensor eigendecomposition described by De Lathauwer et al. in \cite{DeLatCasCar} which is prevalent in signal processing applications, namely in, the blind identification of underdetermined mixtures problems. Last, we describe multilinear systems in PDEs and quantum models. We provide numerical methods for solving multilinear systems of PDEs and tensor eigenvalue decompositions for high dimensional eigenvalue problems. Multilinear systems which do not fit in the framework are addressed by providing pseudo-inversion methods.

%
%
%
%
%

\section{Preliminaries}

We denote the scalars in $\mathbb{R}$ with lower-case letters $(a,b,\ldots)$ and the vectors with bold lower-case letters $(\bf{a},\bf{b},\ldots)$.  The matrices are written as bold upper-case letters $(\bf{A}, \bf{B},\ldots)$ and the symbol for tensors are calligraphic letters $(\mathcal{A},\mathcal{B},\ldots)$. The subscripts represent the following scalars:  $\mathcal{(A)}_{ijk}=a_{ijk}$, $(\bold{A})_{ij}=a_{ij}$, $(\bold{a})_i=a_i$. The superscripts indicate the length of the vector or the size of the matrices. For example, $\bold{b}^{K}$ is a vector with length $K$ and $\bold{B}^{N \times K}$ is a $N \times K$ matrix. In addition, the lower-case superscripts on a matrix indicate the mode in which has been matricized. 

The order of a tensor refers to the cardinality of the index set.  A matrix is a second-order tensor and a vector is a first-order tensor.

\begin{definition}[even and odd tensors]
Given an $N$th tensor $\mathcal{T} \in \mathbb{R}^{I_1 \times I_2 \times \ldots \times I_N}$.
If $N$ is even (odd), then $\mathcal{T}$ is an even (odd) $N$th order tensor.
\end{definition}

\begin{definition}[Einstein product \cite{Einstein}]
For any $N$, the Einstein product  is defined by the operation $\ast_N$ via
\begin{eqnarray}\label{eins2}
(\mathcal{A} \ast_N \mathcal{B})_{i_1\ldots i_N k_{N+1} \ldots k_M} = \sum_{k_1\ldots  k_N} a_{i_1i_2 \ldots i_N k_1 \ldots k_N} b_{k_1 \ldots k_N k_{N+1} k_{N+2} \ldots k_M}. 
\end{eqnarray}
where $\mathcal{A} \in  \mathbb{T}_{I_1,\ldots,I_N,K_1,\ldots,K_N}(\mathbb{R})$ and $\mathcal{B} \in  \mathbb{T}_{K_1,\ldots,K_N,K_{N+1},\ldots,K_M}(\mathbb{R})$.
\end{definition}

For example, if $\mathcal{T}, \mathcal{S} \in \mathbb{R}^{I \times J \times I \times J}$, the operation $\ast_2$ is defined by the following:
\begin{eqnarray}\label{eins1}
(\mathcal{T} \ast_2 \mathcal{S})_{i j\hat{i}\hat{j}} = \sum_{u=1}^{I}\sum_{v=1}^{J} t_{i j  u v} s_{u v \hat{i}\hat{j}}.
\end{eqnarray}

The Einstein product is a contracted product that it is widely used in the area of continuum mechanics \cite{LaiRubinKrempl} and ubiquitously appears in the study of the theory of relativity \cite{Einstein}.  
Notice that the Einstein product $\ast_1$ is the usual matrix multiplication since
\begin{eqnarray}\label{einsmat}
({\bf{M}} \ast_1 {\bf{N}})_{ij} = \sum_{k=1}^{K} m_{i k } n_{k j} = ({\bf{M}} {\bf{N}})_{ij} 
\end{eqnarray}
for ${\bf{M}} \in \mathbb{R}^{I \times K}, {\bf{N}} \in \mathbb{R}^{K \times J}$. 

\begin{definition}[Tucker mode-$n$ product]
Given a tensor $\mathcal{T} \in \mathbb{R}^{I \times J \times K}$ and the matrices $\bold{A} \in \mathbb{R}^{\hat{I} \times I}$, $\bold{B} \in \mathbb{R}^{\hat{J} \times J}$ and $\bold{C}\in \mathbb{R}^{\hat{K} \times K}$, then the Tucker mode-$n$ products are the following:
{\small \begin{eqnarray*}
(\mathcal{T} \bullet_1 \bold{A})_{\hat{i},j,k} &=& \sum_{i=1}^{I} t_{ijk}a_{\hat{i}i},~\forall \hat{i},j,k~\hspace{.15cm}\mbox{(mode-1 product)}\\
(\mathcal{T} \bullet_2 \bold{B})_{\hat{j},i,k} &=& \sum_{j=1}^{J} t_{ijk}b_{\hat{j}j},~\forall \hat{j},i,k~\hspace{.15cm}\mbox{(mode-2 product)}\\
(\mathcal{T} \bullet_3 \bold{C})_{\hat{k},i,j} &=& \sum_{k=1}^{K} t_{ijk}c_{\hat{k}k},~\forall \hat{k},i,j~\hspace{.15cm}\mbox{(mode-3 product)}
\end{eqnarray*}}
\end{definition}
Notice that the Tucker product $\bullet_n$ is the Einstein product $\ast_1$ in which the mode summation is specified.

\begin{figure}[h]
\begin{center}
		{\label{matunfolding}}\includegraphics[width = 90 mm]{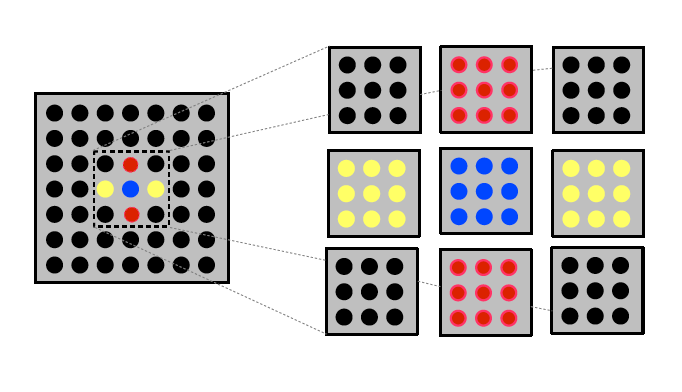}
	\caption{Matrix representation of $\mathcal{S} \in \mathbb{R}^{I_1 \times I_2 \times I_3 \times I_4}$ where $I_1=3, I_2=3, I_3=7, I_4=7$ with $3 \times 3$ matrix slices. There are $7 \cdot 7$ total $3 \times 3$ matrix slices. Here are nine matrix slices with the indices fixed at $(i_3,i_4)$:  $\bold{S}^{(3,4)}_{i_3=3,i_4=3}, \bold{S}^{(3,4)}_{i_3=3,i_4=4}, \bold{S}^{(3,4)}_{i_3=3,i_4=5}$ (top row, right), $\bold{S}^{(3,4)}_{i_3=4,i_4=3}, \bold{S}^{(3,4)}_{i_3=4i_4=4}, \bold{S}^{(3,4)}_{i_3=4,i_4=5}$ (middle row, right), $\bold{S}^{(3,4)}_{i_3=5,i_4=3}, \bold{S}^{(3,4)}_{i_3=5,i_4=4}, \bold{S}^{(3,4)}_{i_3=5,i_4=5}$ (bottom row, right)}
	\end{center}
\end{figure} 

The definitions below describe the representation of higher-order tensors into matrices.

\begin{definition}[Matrix and subtensor slices]\label{matricize1}
A third-order tensor $\mathcal{S} \in \mathbb{R}^{I \times J \times K}$ has three types of matrix slices obtained by fixing the index of one of the modes.  The matrix slices of  $\mathcal{S} \in \mathbb{R}^{I \times J \times K}$ are the following: $\bold{S}^{1}_{i={\alpha}} \in \mathbb{R}^{J \times K}$ with fixed $i={\alpha}$,  $\bold{S}^{2}_{j={\alpha}} \in \mathbb{R}^{I \times K}$ with fixed $j={\alpha}$ and  $\bold{S}^{3}_{k={\alpha}} \in \mathbb{R}^{I \times J}$ with fixed $k={\alpha}$. For a Nth-order tensor $\mathcal{S} \in \mathbb{R}^{I_1 \times I_2 \times I_3 \times \ldots \times I_N}$, the subtensors are the $(N-1)$th-order tensors denoted by $\mathcal{S}^{n}_{i_n=\alpha} \in \mathbb{R}^{I_1 \times I_2 \times I_3 \times \ldots \times  I_{n-1} \times I_{n+1} \times \ldots \times I_N}$ which are obtained by fixing the index of the $n$th mode. 
\end{definition}

\begin{definition}[Matrix slices with several indices fixed]\label{matricize2}
A fourth-order $\mathcal{S} \in \mathbb{R}^{I_1 \times I_2 \times I_3 \times I_4}$ has six types of matrix slices by fixing two indices.  A matrix slice of $\mathcal{S} \in \mathbb{R}^{I_1 \times I_2 \times I_3 \times I_4}$ is  $\bold{S}^{(3,4)}_{i_3={\alpha},i_4={\beta}} \in \mathbb{R}^{I_1 \times I_2}$ with $i_3=\alpha$ and $i_4=\beta$ fixed.  In general, for any $N$th order tensor $\mathcal{S} \in \mathbb{R}^{I_1 \times I_2 \times I_3 \times \ldots \times I_N}$, there are $\binom{N}{N-2}$ different matrix slices by holding $N-2$. A matrix slice of $\mathcal{S} \in \mathbb{R}^{I_1 \times I_2 \times I_3 \times \ldots \times I_N}$  is  $\bold{S}^{(3,4,\ldots,N)}_{i_3={\alpha_3},i_4={\alpha_4},\ldots,i_N={\alpha_N}} \in \mathbb{R}^{I_1 \times I_2}$ with indices $i_3,\ldots,i_n$ fixed. The subscripts $(3,4)$ and  $(3,4,\ldots,N)$ indicate which indices are fixed. Moreover, $(\mathcal{S})_{i_1i_2i_3i_4}=(\bold{S}^{(3,4)}_{i_3,i_4})_{i_1i_2}$. 
\end{definition} 
This definition is different from Definition \ref{matricize1} since several indices are fixed at a time; see Figure $1$. The matrix representation in Figure $1$ is consistent with the matrix representation in Matlab where the last $N-2$ indices are fixed.

\section{Tensor Group Structure and Decompositions}

For the sake of clarity, the main discussion is limited to fourth-order tensors, although all definitions and theorems hold for any even high-order tensors.  Here a group structure on a set of fourth order tensor through a push-forward map on the general linear group is defined.  Also several consequential results from the group structure will be discussed. 


\begin{definition}[Binary operation]
A binary operation $\star$ on a set $\mathbb{G}$ is a rule that assigns to each ordered pair $(\mathscr{A},\mathscr{B})$ of elements of $\mathbb{G}$ some element of $\mathbb{G}$.
\end{definition}

\begin{definition} \label{groupdef}
A group $(\mathbb{G}, \star)$ is a set $\mathbb{G}$, closed under a binary operation $\star$, such that the following axioms are satisfied:
\begin{itemize}
\item[$(A1)$]
The binary operation $\star$ is associative; i.e. $(\mathscr{A} \star \mathscr{B}) \star \mathscr{C} = \mathscr{A} \star (\mathscr{B} \star \mathscr{C})$ for $\mathscr{A}, \mathscr{B}, \mathscr{C} \in \mathbb{G}$.

\item[$(A2)$]
There is an element $\mathscr{E} \in \mathbb{G}$ such that $\mathscr{E} \star \mathscr{X} = \mathscr{X} \star \mathscr{E}$ for all $\mathscr{X} \in \mathbb{G}$.  This element $\mathscr{E}$ is an identity element for $\star$ on $\mathbb{G}$. 

\item[$(A3)$]
For each $\mathscr{A} \in \mathbb{G}$, there is an element $\widetilde{\mathscr{A}} \in \mathbb{G}$ with the property that 
$\widetilde{\mathscr{A}} \star \mathscr{A}= \mathscr{A} \star \widetilde{\mathscr{A}} = \mathscr{E}$.

\end{itemize}
\end{definition}

\begin{definition}[Transformation]
Let $\mathcal{A} \in \mathbb{T}_{I_1,I_2,I_1,I_2}(\mathbb{R})$ and  $\bold{A} \in \mathbb{M}_{I_1I_2,I_1I_2}(\mathbb{R})$. 
Then the transformation $f: \mathbb{T}_{I_1,I_2,I_1,I_2}(\mathbb{R}) \longrightarrow  \mathbb{M}_{I_1I_2,I_1I_2}(\mathbb{R})$ with $f(\mathcal{A})=\bold{A}$ is defined component-wise as
\begin{eqnarray}\label{transf1}
(\mathcal{A})_{i_1i_2i_1i_2} \xrightarrow[]{f} (\bold{A})_{[i_1+(i_2-1)I_1][i_1+(i_2-1)I_1]}
\end{eqnarray}
If $\mathcal{A} \in \mathbb{T}_{I_1,I_2,I_3,I_4}(\mathbb{R})$ and $\bold{A} \in \mathbb{M}_{I_1I_2,I_3I_4}(\mathbb{R})$, then
\begin{eqnarray}\label{transf2}
(\mathcal{A})_{i_1i_2i_3i_4}  \xrightarrow[]{f}  (\bold{A})_{[i_1+(i_2-1)I_1][i_3+(i_4-1)I_3]}.
\end{eqnarray}
Moreover, the transformation is
\begin{eqnarray}\label{transfN}
(\mathcal{A})_{i_1 i_2 \ldots  i_N  j_1 j_2 \ldots  j_N}\xrightarrow[]{f} (\bold{A})_{[i_1 + \sum_{k=2}^{N} (i_k -1) \prod_{l=1}^{k-1} I_l  ][j_1 + \sum_{k=2}^{N} (j_k -1) \prod_{l=1}^{k-1} J_l]}
\end{eqnarray}
when $\mathcal{A} \in \mathbb{T}_{I_1,\ldots,I_N,J_1,\ldots,J_N}(\mathbb{R})$ and $\bold{A} \in \mathbb{M}_{I_1\ldots I_N,J_1\ldots J_N}(\mathbb{R})$.
\end{definition}

These transformations which are known as column (row) major format in many computer languages are typically used to enhance efficiency in accessing arrays. This mapping (\ref{transf1}) is commonly used in matricization of fourth order tensors in signal processing applications; e.g. see \cite{DeLatCasCar}.

\begin{lemma}\label{lemmaf}
Let $f$ be the map defined in (\ref{transf1}).  Then the following properties hold:
\begin{enumerate}
\item
The map $f$ is a bijection. Moreover, there exists a bijective inverse map $f^{-1}: \mathbb{M}_{I_1I_2,I_1I_2}(\mathbb{R}) \rightarrow \mathbb{T}_{I_1,I_2,I_1,I_2}(\mathbb{R})$.
\item
The map satisfies $f(\mathcal{A} \ast_2 \mathcal{B})= f(\mathcal{A}) \cdot f(\mathcal{B})$ where '$\cdot$' refers to the usual matrix multiplication
\end{enumerate}
\end{lemma}
\begin{proof}
\begin{enumerate}
\item[(1)]
According to the definition of $f$, we can define a map $h: I_1 \times I_2 \rightarrow I_{1}I_{2}$ by $h(i_1, i_2)=i_{1}+(i_{2} -1)I_1$ where $I_k=\{1,\ldots,I_k\}$ and $I_{k}I_{l}=\{1,\ldots,I_kI_l\}$. Clearly, the map $h$ is a bijection so it follows $f$ is a bijection. 
\item[(2)]
Since $f$ is a bijection, for some $1 \leq i,j \leq I_1 I_2$, there exists unique $i_1, i_2, j_1, j_2,$ for $1 \leq i_1, j_1 \leq I_1, 1 \leq i_2, j_2 \leq I_2$ such that $(i_2 -1)I_1 + i_1=i$, and $(j_2 -1)I_1 + j_1=j$. So,
 $$[f(\mathcal{A} \ast_2 \mathcal{B})]_{i j} = (\mathcal{A} \ast_2 \mathcal{B})_{i_1 i_2 j_1 j_2}= \sum_{u,v}a_{i_1 i_2 u v}b_{u v j_1 j_2}$$
 $$ [f(\mathcal{A}) \cdot f(\mathcal{B})]_{i j}= \sum_{r=1}^{I_1 I_2}[f(\mathcal{A})]_{i r}[f(\mathcal{B})]_{r j}.$$
 For every $1 \leq r \leq I_1 I_2$, there exists unique $u,v$ such that $(u-1)I_1+v=r$. So, 
 $$\sum_{u,v}a_{i_1 i_2 u v}b_{u v j_1 j_2}= \sum_{r=1}^{I_1 I_2}[f(\mathcal{A})]_{i r}[f(\mathcal{B})]_{r j}.$$
\end{enumerate}
\end{proof}

It follows from the properties of $f$ that the Einstein product (\ref{eins1}) can be defined through the transformation:
\begin{eqnarray}\label{productmap}
\mathcal{A} \ast_2 \mathcal{B}=f^{-1} [f(\mathcal{A} \ast_2 \mathcal{B})]= f^{-1}[f(\mathcal{A}) \cdot f(\mathcal{B})].
\end{eqnarray}
Consequently, the inverse map $f^{-1}$ satisfies 
\begin{eqnarray}\label{inversehomor}
f^{-1}(\bold{A}\cdot  \bold{B})= f^{-1}(\bold{A})  \ast_2 f^{-1}(\bold{B}).
\end{eqnarray}

Recall that a subset of $\mathbb{M}_{N,N}(\mathbb{R})$ consisting of all invertible $N \times N$ matrices with the matrix multiplication is a group. Let the subset $\mathbb{M} \subset  \mathbb{M}_{I_1I_2,I_1I_2}(\mathbb{R})$ contain all invertible  $I_1I_2 \times I_1I_2$. Define $\mathbb{T}=\{ \mathcal{T} \in  \mathbb{T}_{I_1,I_2,I_1,I_2}(\mathbb{R}): det(f(\mathcal{T})) \neq 0 \}$ where  $\bold{T} \in \mathbb{M}$ with $\bold{T}=f(\mathcal{T})$. 

\begin{thm} \label{c'estgroup}
Suppose $(\mathbb{M},\cdot)$ is a group. Let $f: \mathbb{T} \rightarrow \mathbb{M}$ be any bijection.  Then we can define a group structure on $\mathbb{T}$ by defining
\[  \mathcal{A} \ast_2 \mathcal{B} =f^{-1}[f(\mathcal{A}) \cdot f(\mathcal{B})] \]
for all $\mathcal{A}, \mathcal{B} \in \mathbb{T}$. In other words, the binary operation $\ast_2$ satisfies the group axioms. Moreover, the mapping $f$ is an isomorphism.
\end{thm}

The proof is straightforward as it is in the matrix analogue. See the details of the proof in the Appendix. Moreover, the group structure can be cast a ring isomorphism. Further discussion or requirement of the ring structure is not needed hereafter so we do not include the proof. 

\begin{cor} \label{corforN}
Define $\widetilde{\mathbb{T}}=\{ \mathcal{T} \in  \mathbb{T}_{I_1,\ldots,I_N,J_1,\ldots,J_N}(\mathbb{R}), det(f(\mathcal{T})) \neq 0 \}$ with $I_m=J_m$ for $m=1,\ldots,N$.
Then the ordered pair $\widetilde{\mathbb{T}}$ is a group where the operation $\ast_N$ is the generalized Einstein product in (\ref{eins2}).
\end{cor}
\begin{proof}
The generalization of the transformation $f$ (\ref{transfN}) on the set $\widetilde{\mathbb{T}}$ with the binary operation $\ast_N$ easily provide the extension for this case.
\end{proof}

\begin{thm} \label{pasgroup}
The ordered pair $(\widehat{\mathbb{T}},\ast_{N})$ where $\widehat{\mathbb{T}}=\{ \mathcal{T} \in \mathbb{T}_{I_1, I_2 \ldots ,I_{2N-1}}   \}$ is not a group under the operation $\ast_{N}$.
\end{thm}
\begin{proof}
Take $N=2$. Then $\mathcal{T} \ast_2 \mathcal{S} \notin \widehat{\mathbb{T}}$ where $\mathcal{T}, \mathcal{S} \in \mathbb{T}_{I_1, I_2, I_3}$. It follows that $\widehat{\mathbb{T}}$ is not closed under $\ast_2$. Thus the ordered pair $(\widehat{\mathbb{T}},\ast_2)$ is not a group.
\end{proof}

Theorem (\ref{pasgroup}) implies that odd order tensors have no inverses with respect to the operation $\ast_N$, although such binary operation may exist in which the set of odd order tensors exhibits a group structure. Lemma (\ref{lemmaf}) and Theorem (\ref{c'estgroup}) show that the transformation $f$ is an isomorphism between groups $\mathbb{T}$ and $\mathbb{M}$.  From Corollary (\ref{corforN}), it follows that these structural properties are preserved for any ordered pair  $(\widetilde{\mathbb{T}},\ast_N)$ for any $N$. Thus in the following Section $3.2$, some properties and applications of the tensor group structure are addressed.   

Tensors $\mathcal{T} \in \mathbb{T}_{I_1,I_2,I_3,I_4}(\mathbb{R})$ with different mode lengths which are similar to rectangular matrices have no inverses under $\ast_2$. In Section 6, we discuss \emph{pseudo-inverses} for odd order tensors and even order tensors with distinct mode lengths.

\subsection{Decompositions via Isomorphic Group Structures}

Theorem \ref{c'estgroup} implies that $(\mathbb{T},\ast_2)$ is structurally similar to $(\mathbb{M},\cdot)$. Thus we endow  $(\mathbb{T},\ast_2)$ with the group structure such  $(\mathbb{T},\ast_2)$ and $(\mathbb{M},\cdot)$ are isomorphic as groups. This section discusses some of the definitions, theorems and decompositions preserved by the transformation.

\begin{definition}[Transpose] \label{tensortranspose}
The transpose of $\mathcal{S}  \in \mathbb{R}^{I \times J \times I \times J}$ is a tensor $\mathcal{T}$ which has entries $t_{ijkl}=s_{klij}$. We denote the transpose of $\mathcal{S}$ as $\mathcal{T}=\mathcal{S}^T$.  If $\mathcal{S}  \in \mathbb{R}^{I_1 \times I_2 \ldots \times I_N \times J_1 \times \ldots \times J_N}$, then
$(\mathcal{T})_{i_1,i_2,\ldots, i_n,j_1,j_2,\ldots,j_n}=(\mathcal{S})^T_{j_1,j_2,\ldots, j_n,i_1,i_2,\ldots,i_n}$ is the transpose of $\mathcal{S}$.
\end{definition}
 
\begin{definition}[Symmetric tensor]\label{symmetricT}
A tensor $\mathcal{S}  \in \mathbb{R}^{I_1 \times I_2 \ldots \times I_N \times J_1 \times \ldots \times J_N} $ is symmetric if $\mathcal{S}=\mathcal{S}^T$, that is, $s_{i_1,i_2,\ldots, i_n,j_1,j_2,\ldots,j_n}=s_{j_1,j_2,\ldots, j_n,i_1,i_2,\ldots,i_n}$.  
\end{definition}
 
\begin{definition}[Orthogonal tensor]\label{orthogonaltensor}
A tensor $\mathcal{U} \in \mathbb{R}^{I \times J \times I \times J}$ is orthogonal if $\mathcal{U}^T \ast_2 \mathcal{U}= \mathcal{I}$
where $\mathcal{I}$ is the identity tensor under the binary operation $\ast_2$.
\end{definition}

\begin{definition}[Identity tensor]
The identity tensor $\mathcal{E}$ is 
\begin{eqnarray*}
\mathscr{E}_{i_1 i_2 j_1 j_2}=\delta_{i_1 j_1}\delta_{i_2 j_2}
\end{eqnarray*}
where
\begin{eqnarray*}
\delta_{lk}=
\begin{cases}
$1,$ & \mbox{$l = k$} \\
$0,$ & \mbox{$l \neq k$}.
\end{cases}
\end{eqnarray*}
It generalizes to an $2N$th order identity tensor,
\begin{eqnarray}\label{identity}
(\mathcal{I})_{i_1 i_2 \ldots i_N j_1 j_2 \ldots j_N}=\prod_{k=1}^N\delta_{i_k j_k}.
\end{eqnarray}
\end{definition}

\begin{definition}[Diagonal tensor]\label{diagonal}
A tensor $\mathcal{D} \in \mathbb{R}^{I \times J \times I \times J}$ is called diagonal if $d_{ijkl}=0$ when $i \neq k$ and $j \neq l$.
\end{definition}

The diagonal tensor $\mathcal{D} \in \mathbb{R}^{I_1 \times \ldots \times I_N}$ in tensor decompositions like in Parallel Factorization and Canonical Decomposition\cite{CarolChang,Harshman} has nonzero entries $d_{i_1,i_2,\ldots, i_n}$ when $i_1= \ldots =i_N$. Definition \ref{diagonal} has in general more non-zero entries than the usual definition.  This definition is consistent with the identity tensor (\ref{identity}), that is, the diagonal and the identity tensors have nonzero entries on the same indices.

 \begin{thm}[Singular value decomposition (SVD)]\label{TSVD}
 Let $\mathcal{A} \in \mathbb{R}^{I \times J \times I \times J}$ with $R=rank(f(\mathcal{A}))$. The singular value decomposition for tensor $\mathcal{A}$ has the form
 \begin{eqnarray}\label{BLNTsvd}
 \mathcal{A}=\mathcal{U} \ast_2 \mathcal{D} \ast_2 \mathcal{V}^T 
 \end{eqnarray}
 where $\mathcal{U} \in \mathbb{R}^{I \times J \times I \times J}$ and $\mathcal{V} \in \mathbb{R}^{I \times J \times I \times J}$ are orthogonal tensors and $\mathcal{D} \in \mathbb{R}^{I \times J \times I \times J}$
is a diagonal tensor with entries $\sigma_{ijij}$ called \emph{singular values}. Moreover, the decomposition (\ref{BLNTsvd}) can be written as 
\begin{eqnarray} \label{BLNTsvdsum}
\mathcal{A}=\sum_{kl} \sum_{ij,\hat{i}\hat{j}}  \sigma_{klkl}  (\bold{U}_{kl}^{(3,4)})_{ij}  \circ  (\bold{V}_{kl}^{(3,4)})_{\hat{i}\hat{j}}, 
\end{eqnarray}
a sum of fourth order tensors. The matrices $\bold{U}_{ij}^{(3,4)}$ and $\bold{V}_{ij}^{(3,4)}$ are called left and right singular matrices.
\end{thm}
The symbol $\circ$ denotes the outer product where $\mathcal{A}_{ijkl}=\bold{B}_{ij} \circ \bold{C}_{kl}=\bold{B}_{ij}\bold{C}_{kl}$. Recall from Definition 
\ref{matricize2} that $\bold{U}_{ij}^{(3,4)}$ and $\bold{V}_{ij}^{(3,4)}$ are matricizations of fourth order tensors $\mathcal{U}$ and $\mathcal{V}$, respectively.\\ 
\\
\begin{proof}
Let $\bold{A}=f(\mathcal{A})$. From the isomorphic property (\ref{inversehomor}) and Theorem \ref{c'estgroup}, we have 
$$
\bold{A}=\bold{U}\cdot \bold{D} \cdot \bold{V}^T    \xrightarrow[]{f^{-1}}   \mathcal{A}=\mathcal{U} \ast_2 \mathcal{D} \ast_2 \mathcal{V}^T 
$$
where $\bold{U}$ and $\bold{V}$ are orthogonal matrices and $\bold{D}$ is a diagonal matrix. In addition, $\bold{U}\cdot \bold{U}^T=\bold{I}$ and $\bold{V}\cdot \bold{V}^T=\bold{I}  \xrightarrow[]{f^{-1}} \mathcal{U}^T \ast_2 \mathcal{U}= \mathcal{I}$ and $\mathcal{V}^T \ast_2 \mathcal{V}= \mathcal{I}$.
\end{proof}

 \begin{thm}[Eigenvalue decomposition(EVD) for symmetric tensor]\label{TEVD}
Let $\bar{\mathcal{A}} \in \mathbb{R}^{I \times J \times I \times J}$ and $R=rank(f(\mathcal{A}))$. $\bar{\mathcal{A}}$ is a real symmetric tensor if and only if there is a real orthogonal matrix $\mathcal{P} \in  \mathbb{R}^{I \times J \times I \times J}$ and a real diagonal matrix $\bar{\mathcal{D}} \in  \mathbb{R}^{I \times J \times I \times J}$ such that
\begin{eqnarray}\label{eigdecomp} 
\bar{\mathcal{A}}=\mathcal{P} \ast_2 \bar{\mathcal{D}} \ast_2 \mathcal{P}^T
\end{eqnarray}
where $\mathcal{P} \in \mathbb{R}^{I \times J \times I \times J}$ is an orthogonal tensor and $\bar{\mathcal{D}} \in \mathbb{R}^{I \times J \times I \times J}$
is a diagonal tensor with entries $\bar{\sigma}_{ijij}$ called \emph{eigenvalues}. Moreover, the decomposition (\ref{eigdecomp}) can be written as 
\begin{eqnarray} \label{specdecompsum}
\bar{\mathcal{A}}=\sum_{kl} \sum_{ij\hat{i}\hat{j}} \bar{\sigma}_{klkl}  (\bold{P}_{kl}^{(3,4)})_{ij}  \circ  (\bold{P}_{kl}^{(3,4)})_{\hat{i}\hat{j}}, 
\end{eqnarray}
a sum of fourth order tensors. The matrix $\bold{P}_{kl}^{(3,4)} \in \mathbb{R}^{I \times J}$ is called an eigenmatrix.
\end{thm}
\begin{proof}
From the isomorphic property (\ref{inversehomor}) and Theorem \ref{c'estgroup}, we obtain that there exist some orthogonal matrix $\mathcal{P}$ and diagonal $\bar{\mathcal{D}}$ such that $\bar{\mathcal{A}}=\mathcal{P} \ast_2 \bar{\mathcal{D}} \ast_2 \mathcal{P}^T$. Moreover, the fourth order tensor $\hat{\mathcal{P}}_{ij\hat{i}\hat{j}}=(\bold{P}_{kl}^{(3,4)})_{ij}  \circ  (\bold{P}_{kl}^{(3,4)})_{\hat{i}\hat{j}}$ is symmetric since 
$\hat{\mathcal{P}}_{ij\hat{i}\hat{j}}=\sum_{ij,\hat{i}\hat{j}} (\bold{P}_{kl}^{(3,4)})_{ij}  \circ  (\bold{P}_{kl}^{(3,4)})_{\hat{i}\hat{j}}=\sum_{ij,\hat{i}\hat{j}} \mathcal{P}_{ijkl}\mathcal{P}_{\hat{i}\hat{j}kl}^T=\sum_{s}\bold{P}_{rs}\cdot \bold{P}_{\hat{r}s}^T=\sum_s \bold{P}_{\hat{r}s}\cdot \bold{P}_{rs}^T=\sum_{ij,\hat{i}\hat{j}} \mathcal{P}_{\hat{i}\hat{j}kl}\mathcal{P}_{ijkl}^T=(\bold{P}_{kl}^{(3,4)})_{\hat{i}\hat{j}} \circ (\bold{P}_{kl}^{(3,4)})_{ij}=\hat{\mathcal{P}}_{\hat{i}\hat{j}ij}.$
\end{proof}
\begin{remark}
If the eigenmatrix $(\bold{P}_{kl}^{(3,4)})$ is symmetric, that is, $(\bold{P}_{kl}^{(3,4)})_{ij}=(\bold{P}_{kl}^{(3,4)})_{ji}$, then the entries of $\bar{\mathcal{A}}$
has the following symmetry: $\bar{a}_{jilk}=\bar{a}_{ijkl}$. If $\bar{a}_{jilk}=\bar{a}_{ijkl}$ and $\bar{a}_{ijkl}=\bar{a}_{klij}$, then (\ref{specdecompsum}) is exactly the tensor eigendecomposition found in the paper of De Lathauwer et al. \cite{DeLatCasCar} when $I=J$. The fourth order tensor in  \cite{DeLatCasCar} is a quadricovariance in the blind identification of underdetermined mixtures problems.
\end{remark}

\subsection{Connections to Standard Tensor Decompositions}
 In 1927, Hitchcock \cite{Hitch1,Hitch2} introduced the idea that a tensor is decomposable into a sum of a finite number of rank-one tensors.  Today, we refer to this decomposition as canonical polyadic (CP) tensor decomposition (also known as CANDECOMP \cite{CarolChang} or PARAFAC \cite{Harshman}). CP is a linear combination of rank-one tensors, i.e.
\begin{eqnarray}\label{CP}
\mathcal{T}= \sum_{r=1}^R a_r \circ b_r \circ c_r \circ d_r
\end{eqnarray}
where $\mathcal{T}\in \mathbb{R}^{I \times J \times K \times L}$, $a_r  \in \mathbb{R}^{I}, b_r \in \mathbb{R}^{J}$, $c_r \in \mathbb{R}^{K}$ and $d_r \in \mathbb{R}^{L}$. The column vectors $a_r, b_r$, $c_r$ and $d_r$ form the so-called factor matrices $\bold{A}$, $\bold{B}$, $\bold{C}$ and $\bold{D}$, respectively.  The tensorial rank \cite{Hitch2} is the minimum $R \in \mathbb{N}$ such that $\mathcal{T}$ can be expressed as a sum of $R$ rank-one tensors. Moreover, in 1977 Kruskal \cite{Kruskal} proved that for third order tensor,
$$
2R +2 \leq rank(\bold{A}) + rank(\bold{B}) + rank(\bold{C})
$$
is the sufficient condition for uniqueness of $\mathcal{T}=\sum_{r=1}^R a_r \circ b_r \circ c_r$ up to permutation and scalings. Kruskal's uniqueness condition was then generalized for $n \geq 3$ by Sidiropoulous and Bro \cite{SidBro}:

\begin{eqnarray}\label{SidBroUnique}
2R + (n-1) \leq \sum_{j=1}^n rank(\bold{A}^{(j)})
\end{eqnarray}
for $\mathcal{T}= \sum_{r=1}^R a^{(1)}_r \circ \cdots \circ a^{(n)}_r$.

Another decomposition called Higher-Order SVD (also known as Tucker and Multilinear SVD) was introduced by Tucker \cite{Tucker1,Tucker2,Tucker3,HOSVD} in which a tensor is decomposable into a core tensor \emph{multiplied} by a matrix along each mode, i.e.
\begin{eqnarray}\label{MSVD}
\mathcal{T}= \mathcal{S} \bullet_1 \bold{A} \bullet_2 \bold{B} \bullet_3 \bold{C} \bullet_4 \bold{D}
\end{eqnarray}
where $\mathcal{T}, \mathcal{S} \in \mathbb{R}^{I \times J \times K \times L}$ are fourth order tensors with four orthogonal factors $\bold{A} \in \mathbb{R}^{I \times I}$, $\bold{B} \in \mathbb{R}^{J \times J}$, $\bold{C} \in \mathbb{R}^{K \times K}$and $\bold{D} \in \mathbb{R}^{L \times L}$. Note that CP can be viewed as a Tucker decomposition where its core tensor $\mathcal{S} \in \mathbb{R}^{I \times J \times K \times L}$ is \emph{diagonal}, that is, the nonzeros entries are located at $(\mathcal{S})_{iiii}$. 

The tensor SVD \ref{BLNTsvd} can be viewed as CP and multilinear SVD.

\begin{lemma}\label{CPSVD}
Let $\mathcal{T} \in \mathbb{R}^{I \times J \times I \times J}$ and $R=rank(f(\mathcal{T}))$. The tensor SVD (\ref{BLNTsvd}) in Theorem \ref{TSVD} is equivalent to CP (\ref{CP}) if there exist $\mathcal{A} \in \mathbb{R}^{I \times I \times J}, \mathcal{B} \in \mathbb{R}^{J \times I \times J} $, $\mathcal{C} \in \mathbb{R}^{I \times I \times J}$ and $ \mathcal{D} \in \mathbb{R}^{J \times I \times J} $ such that $a_{ikl}b_{jkl}=u_{ijkl}$ and $c_{\hat{i}kl}d_{\hat{j}kl}=v_{\hat{i}\hat{j}kl}$.
\end{lemma}
\begin{proof}
Define $r=k + (l-1)I$. Then
\begin{eqnarray*}
\mathcal{T}_{ij\hat{i}\hat{j}}&=&\sum_{kl} \sigma_{klkl}  (\bold{U}_{kl}^{(3,4)})_{ij}  \circ  (\bold{V}_{kl}^{(3,4)})_{\hat{i}\hat{j}}=\sum_{r=1}^{R}   \bar{\sigma}_{rr}  (\bold{U}_{r}^{(3,4)})_{ij}  \circ  (\bold{V}_{r}^{(3,4)})_{\hat{i}\hat{j}}\\
\end{eqnarray*}
where $\sigma_{klkl}=\hat{\sigma}_{rr}$. Since  $u_{ijkl}=(\bold{U}_{kl}^{(3,4)})_{ij}=(\bold{U}_{r}^{(3,4)})_{ij}=a_{ir}b_{jr}$ and  $v_{\hat{i}\hat{j}kl}=(\bold{V}_{kl}^{(3,4)})_{\hat{i}\hat{j}}=(\bold{V}_{r}^{(3,4)})_{\hat{i}\hat{j}}=c_{\hat{i}r}d_{\hat{j}r}$, it follows that
\begin{eqnarray*}
\mathcal{T}_{ij\hat{i}\hat{j}}&=&\sum_{r=1}^{R}  \hat{\sigma}_{rr}  (\bold{U}_{r}^{(3,4)})_{ij}  \circ  (\bold{V}_{r}^{(3,4)})_{\hat{i}\hat{j}}=\sum_{r=1}^{R}  \hat{\sigma}_{rr}  (\bold{U}_{r}^{(3,4)})_{ij}  \circ  (\bold{V}_{r}^{(3,4)})_{\hat{i}\hat{j}}=\sum_{r=1}^{R}   \hat{\sigma}_{rr} a_{ir}b_{jr}c_{\hat{i}r}d_{\hat{j}r}
\end{eqnarray*}
Then,
\begin{eqnarray}\label{CPSVD2}
\mathcal{T}&=&\sum_{r=1}^{R} \bar{\sigma}_{rr} a_{r} \circ b_{r} \circ c_{r} \circ d_{r}.
\end{eqnarray}
Moreover, the factor matrices $\bold{A} \in \mathbb{R}^{I \times IJ}, \bold{B} \in \mathbb{R}^{J \times IJ}, \bold{C} \in \mathbb{R}^{I \times IJ}$ and $\bold{D} \in \mathbb{R}^{J \times IJ}$ are built from concatenating the vectors $a_r$, $b_r$, $c_r$ and $d_r$, respectively.
\end{proof}

\begin{remark}
To satisfy existence and uniqueness of a CP decomposition, the inequality (\ref{SidBroUnique}) must hold; i.e. $2IJ + 3 \leq rank(\bold{A}) + rank(\bold{B}) + rank(\bold{C}) + rank(\bold{D})$. If $R=rank(f(\mathcal{A}))=IJ$, then the decomposition (\ref{CPSVD2}) does not satisfy (\ref{SidBroUnique}). However, if $f(\mathcal{T})$ is sufficiently low rank, that is, $R=rank(f(\mathcal{T})) < IJ$ for some dimensions $I$ and $J$, then (\ref{SidBroUnique}) holds. Futhermore, the existence of the factors $\bold{A}$, $\bold{B}$, $\bold{C}$ and $\bold{D}$ requires that the matricizations,$ \bold{U}_{kl}^{(3,4)}$  and $\bold{V}_{kl}^{(3,4)}$, to be rank-one matrices.
\end{remark}

\begin{lemma}\label{TSVDtoMSVD}
Let $\mathcal{T} \in \mathbb{R}^{I \times J \times I \times J}$ and $R=rank(f(\mathcal{A}))$. The tensor SVD (\ref{BLNTsvd}) in Theorem \ref{TSVD} is equivalent to multilinear SVD (\ref{MSVD}) if there exist $\bold{A} \in \mathbb{R}^{I \times I }, \bold{B} \in \mathbb{R}^{J \times J}$, $\bold{C} \in \mathbb{R}^{I \times I }$ and $\bold{D} \in \mathbb{R}^{J \times J} $ such that $a_{ik}b_{jl}=u_{ijkl}$ and $c_{ik}d_{jl}=v_{ijkl}$.
\end{lemma}
\begin{proof}
From (\ref{BLNTsvdsum}), we have $\mathcal{T}_{ij\hat{i}\hat{j}}=\sum_{kl} \sigma_{klkl}  (\bold{U}_{kl}^{(3,4)})_{ij}  \circ  (\bold{V}_{kl}^{(3,4)})_{\hat{i}\hat{j}}$ which implies $\mathcal{T}_{ij\hat{i}\hat{j}} =\sum_{kl} \sigma_{klkl} a_{ik}b_{jl}c_{\hat{i}k}d_{\hat{j}l}.$
\end{proof}
\begin{remark}
Typically, the core tensor of a multlinear SVD (\ref{MSVD}) is dense. However, the core tensor resulting from Lemma \ref{TSVDtoMSVD} is not dense (possibly sparse); i.e. there are $IJ$ nonzeros elements out of $I^2J^2$ entries in the fourth order core tensor of size $I \times J \times I \times J$.  Similarly, the existence of the decomposition impinges upon the existence of the factors $\bold{A} \in \mathbb{R}^{I \times I }, \bold{B} \in \mathbb{R}^{J \times J}$, $\bold{C} \in \mathbb{R}^{I \times I }$ and $\bold{D} \in \mathbb{R}^{J \times J}$ such that $\mathcal{U}=\bold{A} \circ \bold{B}$ and $\mathcal{V}=\bold{C} \circ \bold{D}$.
\end{remark}

\begin{cor}\label{TEVDtoCP1}
Let $\mathcal{T} \in \mathbb{R}^{I \times J \times I \times J}$ is symmetric and $R=rank(f(\mathcal{A}))$. The tensor EVD \ref{specdecompsum} in Theorem \ref{TEVD} is equivalent to CP (\ref{CP})  if there exist $\mathcal{A} \in \mathbb{R}^{I \times I \times J}, \mathcal{B} \in \mathbb{R}^{J \times I \times J} $ such that $a_{ikl}b_{jkl}=p_{ijkl}$.
\end{cor}
\begin{cor}\label{TEVDtoCP2}
Let $\mathcal{T} \in \mathbb{R}^{I \times J \times I \times J}$ with symmetries $t_{ijkl}=t_{klij}$ and $t_{jikl}=t_{ijkl}$ with $R=rank(f(\mathcal{A}))$. The tensor EVD \ref{specdecompsum} in Theorem \ref{TEVD} is equivalent to CP (\ref{CP})  if there exist $\mathcal{A} \in \mathbb{R}^{I \times I \times J}, \mathcal{B} \in \mathbb{R}^{J \times I \times J} $ such that $a_{ikl}b_{jkl}=p_{ijkl}$.
\end{cor}
\begin{remark}
The CP decomposition from Corollary \ref{TEVDtoCP1} is $\mathcal{T}_{ij\hat{i}\hat{j}}=\sum_{r=1}^{R}   \bar{\sigma}_{rr}  (\bold{P}_{r}^{(3,4)})_{ij}  \circ  (\bold{P}_{r}^{(3,4)})_{\hat{i}\hat{j}} = \sum_{r=1}^{R}  \bar{\sigma}_{rr} a_{ir}b_{jr}a_{\hat{i}r}b_{\hat{j}r}$ with identical factors: $\bold{A}=\bold{C}$ and $\bold{B}=\bold{D}$ from Lemma \ref{CPSVD}. As in Remark 3.17, the existence of the factors $\bold{A}$ and $\bold{B}$ requires that the matricization, $ \bold{P}_{kl}^{(3,4)}$, to be rank-one matrices. In Corollary \ref{TEVDtoCP2}, the added symmetry $t_{jikl}=t_{ijkl}$ implies that $\bold{P}_{kl}^{(3,4)}$ is symmetric (as well as rank-one). Thus, $(\bold{P}_{kl}^{(3,4)})_{ij}=a_{ikl}a_{jkl} \Longrightarrow \mathcal{T}=\sum_{r=1}^{R}  \bar{\sigma}_{rr} a_{ir}a_{jr}a_{\hat{i}r}a_{\hat{j}r}$. This decomposition is known as symmetric CP decomposition \cite{LekHeng}.
\end{remark}
\begin{cor}\label{TEVDtoMSVD}
Let $\mathcal{T} \in \mathbb{R}^{I \times J \times I \times J}$ is symmetric and $R=rank(f(\mathcal{A}))$. The tensor EVD (\ref{specdecompsum}) in Theorem \ref{TEVD} is equivalent to multilinear SVD (\ref{MSVD})  if there exists $\mathcal{A} \in \mathbb{R}^{I \times I \times J}$ such that $a_{ik}b_{jl}=p_{ijkl}$.
\end{cor}
\begin{remark}
The multilinear SVD from Corollary (\ref{TEVDtoMSVD}) is $\mathcal{T}_{ij\hat{i}\hat{j}}=\sum_{kl} \sigma_{klkl}  (\bold{P}_{kl}^{(3,4)})_{ij}  \circ  (\bold{P}_{kl}^{(3,4)})_{\hat{i}\hat{j}}=\sum_{kl} \sigma_{klkl} a_{ik}b_{jl}a_{\hat{i}k}b_{\hat{j}l}$ following Lemma \ref{TSVDtoMSVD}.
\end{remark}

\section{Multilinear Systems}
A multilinear system is a set of $M$ equations with $N$ unknown variables. A linear system is a multilinear system which is conveniently expressed as $\bold{A}\bold{x}=\bold{b}$ where $\bold{A} \in \mathbb{R}^{M \times N}, \bold{x} \in \mathbb{R}^N$ and $\bold{b} \in \mathbb{R}^M$. Similarly, $\mathcal{A} \ast_2 \mathcal{X} = \mathcal{B}$ has $M=I \cdot J \cdot K \cdot L$ equations and $N=R\cdot S\cdot K \cdot L$ unknown variables  if  $\mathcal{A} \in \mathbb{R}^{I \times J \times R \times S}, \mathcal{X} \in \mathbb{R}^{R \times S \times K \times L}$ and $\mathcal{B} \in \mathbb{R}^{I \times J \times K \times L}$. Equivalently, we define a linear transformation $\mathfrak{L}: \mathbb{R}^N \rightarrow \mathbb{R}^M$ such that $\mathfrak{L}(\bold{x})=\bold{A} \bold{x}$ with the property $\mathfrak{L}(c \bold{x} + d \bold{y})=c \mathfrak{L}(\bold{x}) + d \mathfrak{L}(\bold{y})$ for some scalars $c$ and $d$. A bilinear system is defined through $\mathfrak{B}: \mathbb{R}^M \times \mathbb{R}^N \rightarrow \mathbb{R}$ with $\mathfrak{B}(\bold{x},\bold{y})=\bold{y}^T \bold{B} \bold{x}=\bold{B} \bullet_1 \bold{x} \bullet_2 \bold{y}$ where $\bold{B} \in \mathbb{R}^{M \times N}$.  The bilinear map has the linearity properties:
\begin{eqnarray*}
\mathfrak{B}(c \bold{x_1} + d \bold{x_2}, \bold{y})&=&c \mathfrak{B}(\bold{x_1},\bold{y}) + d \mathfrak{B}(\bold{x_2},\bold{y})\\
&\text{and}&\\
\mathfrak{B}(\bold{x}, \bar{c} \bold{y_1} + \bar{d} \bold{y_2})&=&\bar{c} \mathfrak{B}(\bold{x},\bold{y_1}) + \bar{d} \mathfrak{B}(\bold{x},\bold{y_2})
\end{eqnarray*}
for some scalars $c,\bar{c},d,\bar{d}$ and vectors $\bold{x}, \bold{x_1}, \bold{x_2} \in \mathbb{R}^N, \bold{y}, \bold{y_1}, \bold{y_2} \in \mathbb{R}^M$. 

We can define multilinear transformations $\mathcal{M}: \mathbb{R}^{I_1 \times \ldots \times I_N} \rightarrow \mathbb{R}^{J_1 \times \ldots \times J_M}$ for the following multilinear systems:
\begin{itemize}
\item
$\mathcal{B}  \bullet_1 \bold{x} \bullet_2 \bold{y} =b$ where $\mathcal{B} \in \mathbb{R}^{I \times J \times K}$, $\bold{x} \in \mathbb{R}^I$, $\bold{y} \in \mathbb{R}^J$ and $b \in \mathbb{R}$
\item
$\mathcal{M} \ast_2 \bold{X}  \ast_2 \bold{Y}  = b$ where $\mathcal{M} \in \mathbb{R}^{I \times J \times K \times L},$ $\bold{X} \in \mathbb{R}^{K \times L}$,  $\bold{Y} \in \mathbb{R}^{I \times J}$ and  $b \in \mathbb{R}$
\item
$\mathcal{M} \ast_2 \mathcal{X} \ast_3 \mathcal{Y} = \bold{B}$ where $\mathcal{M} \in \mathbb{R}^{I \times J \times K \times L \times M \times N},$ $\mathcal{X} \in \mathbb{R}^{M \times N \times O}$, $\mathcal{Y} \in \mathbb{R}^{K \times L \times O}$ and $\bold{B} \in \mathbb{R}^{I \times J}$. 
\end{itemize}

Multilinear systems model many phenomena in engineering and sciences. In the field of continuum physics and engineering, isotropic and anisotropic elastic models are multilinear systems \cite{LaiRubinKrempl}.  For example,
\[\mathcal{C} \ast_2 \bold{E} =\bold{T}  \]
where $\bold{T}$ and $\bold{E}$ are second order tensors modeling stress and strain, respectively, and the fourth order tensor $\mathcal{C}$ refers to the elasticity tensor. Multilinear systems are also prevalent in the numerical methods for solving partial differential equations (PDEs). To approximate solutions to PDEs, the given continuous problem is typically discretized by using finite element methods  or finite difference schemes to obtain a discrete problem.  The discrete problem is a multilinear system with finitely many unknowns.  

\subsection{Poisson problem with multilinear system solver}
Consider the two-dimensional Poisson problem
 \begin{equation}
\begin{array}{cc}
	-\nabla^2v=f  & \mbox{~in~} \Omega,\\
	u=0                & \mbox{~on~} \Gamma
\label{poisson2d}
\end{array}
\end{equation}
where $\Omega=\{(x,y): 0<x,y<1  \}$ with boundary $\Gamma$, $f$ is a given function and 
\[\nabla^2 v = \frac{\partial^2 v}{\partial x^2}  + \frac{\partial^2 v}{\partial y^2}.\] 
We compute an approximation to the unknown function $v(x,y)$ in (\ref{poisson2d}). Several problems in physics and mechanics are modeled by (\ref{poisson2d}) where the solution $v$ represent, for example, temperature, electro-magnetic potential or  displacement of an elastic membrane fixed at the boundary.

The mesh points are obtained by discretizing the unit square domain with step sizes, $\Delta x$  in the $x$-direction and $\Delta y$ in the $y$-direction; assume $\Delta x =\Delta y$ for simplicity. From the standard central difference approximations, the difference formula,
\begin{equation}
\frac{v_{l-1,m}-2v_{l,m}+v_{l+1,m}}{\Delta x^2}+\frac{v_{l,m-1}-2v_{l,m}+v_{l,m+1}} {\Delta y^2}=f(x_l,y_m),
\label{fdpoisson}
\end{equation}
is obtained. Then the difference equation (\ref{fdpoisson}) is equivalent to
\begin{eqnarray}\label{2dpoisson}
\bold{A_N}\bold{V}+ \bold{V}\bold{A_N}=(\Delta x)^2\bold{F}
\end{eqnarray}
where
\begin{eqnarray}\label{Amat}
\bold{A_N}= \left [
\begin{array}{cccc}
2    & -1         &                & 0\\
-1   & 2          &   \ddots  &   \\
       & \ddots &    \ddots &-1 \\
0     &             &  -1           & 2\\
\end{array} \right ],~
\end{eqnarray}
\begin{eqnarray} \label{VFmat}
\bold{V}= \left [
\begin{array}{cccc}
v_{11}  & v_{12} & \hdots          & v_{1N}     \\
v_{21}  & v_{22} & \ddots          & \vdots       \\
\vdots   & \ddots  & \ddots          & v_{N-1N}  \\
v_{N1} & \hdots  & v_{NN-1}    & v_{NN}\\
\end{array}\right ]~~~~\mbox{and}~~~~
\bold{F}=
\left [
\begin{array}{cccc}
f_{11}  & f_{12} & \hdots          & f_{1N}     \\
f_{21}  & f_{22} & \ddots          & \vdots       \\
\vdots   & \ddots  & \ddots          & f_{N-1N}  \\
f_{N1} & \hdots  & f_{NN-1}    & f_{NN}\\
\end{array}\right ]
\end{eqnarray}
where the entries of $\bold{V}$ and $\bold{F}$ are the values on the mesh on the unit square where $(x_i,y_j)=(i\Delta x, j\Delta x) \in[0,1] \times [0,1]$. Here the Dirichlet boundary conditions are imposed so the values of V are zero at the boundary of the unit square; i.e. $v_{i0}=v_{iN+1}=v_{0,j}=v_{N+1j}=0$ for $0<i,j<N+1$. Typically, $\bold{V}$ and $\bold{F}$ are vectorized which leads to the linear system:
\begin{eqnarray}\label{2dpoissonb}
\bold{A_{N \times N}} \cdot \bold{v}=\left [
\begin{array}{cccc}
\bold{A_N} + 2\bold{I_N}    & -\bold{I_N}         &                & 0\\
-\bold{I_N} & \bold{A_N} + 2\bold{I_N}            &   \ddots  &   \\
       & \ddots &    \ddots &-\bold{I_N}  \\
0     &             &  -\bold{I_N}            & \bold{A_N} + 2\bold{I_N}\\
\end{array} \right ] ~
\left [
\begin{array}{c}
v_{11}\\
v_{12}\\
\vdots\\
v_{NN}
\end{array}\right ]= (\Delta x)^2
\left [
\begin{array}{c}
f_{11}\\
f_{12}\\
\vdots\\
f_{NN}
\end{array}\right ]
\end{eqnarray}
In \cite{Demmel}, the Poisson's equation in two-dimension is expressed as a sum of Kronecker products; i.e.
\begin{eqnarray}\label{kron1}
\bold{A_{N \times N}}=\bold{I_N} \otimes \bold{A_N} + \bold{A_N} \otimes \bold{I_N}.
\end{eqnarray}
The discretized problem in three-dimension is 
\begin{eqnarray}\label{kron2}
 (\bold{A_N} \otimes \bold{I_N} \otimes \bold{I_N} +  \bold{I_N}  \otimes \bold{A_N} \otimes \bold{I_N} +  \bold{I_N} \otimes \bold{I_N} \otimes \bold{A_N}) \cdot vec({\bold{V}})=vec({\bold{F}}).   
\end{eqnarray}
High dimensional Poisson problems are formulated as sums of Kronecker products with vectorized source term and unknowns.

\subsubsection{Higher-Order Tensor Representation}
The higher-order representation of  the 2D discretized Poisson problem (\ref{poisson2d}) is 
\begin{eqnarray}\label{mspoisson2d}
\mathcal{A}_N \ast_2 \bold{V} = \bold{F}
\end{eqnarray}
where $\mathcal{A}_N \in \mathbb{R}^{N \times N \times N \times N}$ and matrices, $\bold{V}$ and $\bold{F}$, are the discretized functions $v$ and $f$ on a unit square mesh defined in (\ref{VFmat}).  The non-zeros entries of the matrix slice $\bold{A_N}^{(3,4)}_{k=\alpha,l=\beta} \in \mathbb{R}^{N \times N}$ are the following:
\begin{eqnarray}\label{fivepoint}
\begin{cases}
(\bold{A_N}^{(3,4)}_{k=\alpha,l=\beta})_{\alpha,\beta}=\frac{4}{(\Delta x)^2}\\
(\bold{A_N}^{(3,4)}_{k=\alpha,l=\beta})_{\alpha-1\beta}=\frac{-1}{(\Delta x)^2}\\
(\bold{A_N}^{(3,4)}_{k=\alpha,l=\beta})_{\alpha+1,\beta}=\frac{-1}{(\Delta x)^2}\\
(\bold{A_N}^{(3,4)}_{k=\alpha,l=\beta})_{\alpha,\beta-1}=\frac{-1}{(\Delta x)^2}\\
(\bold{A_N}^{(3,4)}_{k=\alpha,l=\beta})_{\alpha,\beta+1}=\frac{-1}{(\Delta x)^2}\\
\end{cases}
\end{eqnarray}
for $\alpha,\beta=2,\ldots,N-1.$ These entries form a five-point stencil; see Figure \ref{fig:plot2}. The discretized three-dimensional Poisson equation is
\begin{eqnarray}\label{mspoisson3d}
\bar{\mathcal{A}}_N \ast_3 \mathcal{V} = \mathcal{F}
\end{eqnarray}
where  $\bar{\mathcal{A}}_N \in \mathbb{R}^{N \times N \times N \times N \times N \times N}$ and $\mathcal{V}, \mathcal{F} \in \mathbb{R}^{N \times N \times N}$, containing the values on the discretized unit cube.  Similarly, the entries of the subtensor slice $(\bar{\mathcal{A}}_N)^{(4,5,6)}_{l,m,n} \in \mathbb{R}^{N \times N \times N} $ of  $\bar{\mathcal{A}}_N$ would follow a seven-point stencil; i.e.
\begin{eqnarray}\label{sevenpoint}
\begin{cases}
((\bar{\mathcal{A}}_N)^{(4,5,6)}_{l=\alpha,m=\beta,n=\gamma})_{\alpha,\beta,\gamma}=\frac{6}{(\Delta x)^3}\\
((\bar{\mathcal{A}}_N)^{(4,5,6)}_{l=\alpha,m=\beta,n=\gamma})_{\alpha-1,\beta,\gamma}=\frac{-1}{(\Delta x)^3}\\
((\bar{\mathcal{A}}_N)^{(4,5,6)}_{l=\alpha,m=\beta,n=\gamma})_{\alpha+1,\beta,\gamma}=\frac{-1}{(\Delta x)^3}\\
((\bar{\mathcal{A}}_N)^{(4,5,6)}_{l=\alpha,m=\beta,n=\gamma})_{\alpha,\beta-1,\gamma}=\frac{-1}{(\Delta x)^3}\\
((\bar{\mathcal{A}}_N)^{(4,5,6)}_{l=\alpha,m=\beta,n=\gamma})_{\alpha,\beta+1,\gamma}=\frac{-1}{(\Delta x)^3}\\
((\bar{\mathcal{A}}_N)^{(4,5,6)}_{l=\alpha,m=\beta,n=\gamma})_{\alpha,\beta,\gamma-1}=\frac{-1}{(\Delta x)^3}\\
((\bar{\mathcal{A}}_N)^{(4,5,6)}_{l=\alpha,m=\beta,n=\gamma})_{\alpha,\beta,\gamma+1}=\frac{-1}{(\Delta x)^3}\\
\end{cases}
\end{eqnarray}
for $\alpha,\beta,\gamma=2,\ldots,N-1$ since $v_{ijk}$ satisfies
\[6v_{ijk} -v_{i-1jk} - v_{i+1jk} - v_{ij-1k} -v_{ij+1k} - v_{ijk-1} -v_{ijk+1} = (\Delta x)^3 f_{ijk}.    \]
\begin{figure}[htp]
  \begin{center}
    \subfigure[5-point Stencil]{\label{fig:stencil1}\includegraphics[width = 25 mm]{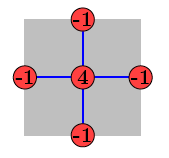}}
    \subfigure[7-point Stencil]{\label{fig:stencil2}\includegraphics[width = 42 mm]{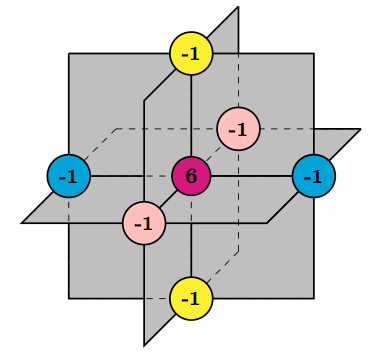}}
    \end{center}
  \caption{Stencils for higher-order tensors.}
  \label{fig:plot2}
\end{figure}
Multilinear systems like (\ref{mspoisson2d}) and (\ref{mspoisson3d}) are the tensor representation of  high dimensional Poisson problems.

\subsection{Iterative Methods}
Here we discuss some methods for solving multilinear system. A naive approach is the Gauss-Newton algorithm for approximating $\mathcal{A}^{-1}$ through the function,
\[ g(\mathcal{X})=\mathcal{A} \ast_2 \mathcal{X} -\mathcal{I} =0  \]
where $\mathcal{I}$ is the identity fourth-order tensor defined in (\ref{identity}) and $\mathcal{X}$ is the unknown tensor. This method is highly inefficient due to a very expensive inversion of a Jacobian. 

To save memory and operational costs, we consider iterative methods for solving multilinear systems. The pseudo-codes in Table \ref{fig:algorithms} describe the biconjugate gradient (BiCG) method for solving multilinear system, $\mathcal{A}\ast_2 \mathcal{X}=\mathcal{B}$, without matricizations. Recall that  the BiCG method requires symmetric and positive definite matrix so that the multilinear system is premultiplied by its transpose $\mathcal{A}^T$ which is defined in Section $3$. 
The BiCG method solves multilinear system by searching along $\mathcal{X}_k=\mathcal{X}_{k-1}+ \alpha_{k-1} \mathcal{P}_{k-1}$ with a line parameter $\alpha_{k-1}$ and a search direction $\mathcal{P}_{k-1}$ while minimizing the objective function $\phi(\mathcal{X}_k + \alpha_{k-1} \mathcal{P}_k)$ where $\phi(\mathcal{X})=\frac{1}{2}\mathcal{X}^T \ast_2 \mathcal{A} \ast_2 \mathcal{X} - \mathcal{X}^T \ast_2 \mathcal{B}$.  It follows that $\phi(\widehat{\mathcal{X}})$ attains a minimum iteratively and precisely at an optimizer $\widehat{\mathcal{X}}$ where $\mathcal{A}\ast_2 \widehat{\mathcal{X}}=\mathcal{B}$. 

The higher-order Jacobi method is also implemented for comparison. The Jacobi method for tensors is an iterative method based on \emph{splitting} the tensor into its diagonal entries from the lower and upper diagonal entries. 
In Figure \ref{fig:poissonplot}, we approximate the solution to the multilinear system (\ref{mspoisson2d}) using two multilinear iterative methods: higher-order biconjugate gradient and Jacobi methods. See Table \ref{fig:algorithms} for the pseudo-codes of the algorithms. In Figure \ref{fig:poissonplot}, BiCG converged faster than Jacobi with fewer number of iterations.  The convergence of Jacobi is slow since the spectral radius with respect to the Poisson's equation is near one \cite{Demmel}. The approximation in Figure \ref{fig:poissonplot} is first order accurate. 

Formulating the discretized Poisson equation in terms of higher-order tensors is convenient since its entries follow a stencil format in Figure \ref{fig:plot2}. The boundary conditions are easily imposed without rearrangements of entries. Also the unknown $v$ is solved on a higher-order mesh; no vectorization is needed. The multilinear system representation has the potential to become a reliable solver of PDEs in very high dimension. For example, implementation of new tensor decompositions which reduce the number of tensor modes are required for higher dimensional problems. The use of low rank preconditioner in tensor format can dramatically increase convergence rates in iterative methods as in the case for sparse linear large systems \cite{Bramley}.

\begin{table}[htp]
  \begin{center}
    \subfigure[Higher-Order Biconjugate Gradient]{\label{fig:hobicg}\includegraphics[width = 75 mm]{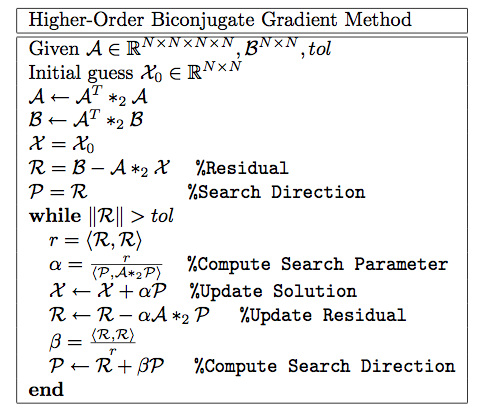}}
    \subfigure[Higher-Order Jacobi]{\label{fig:hoj}\includegraphics[width = 75 mm]{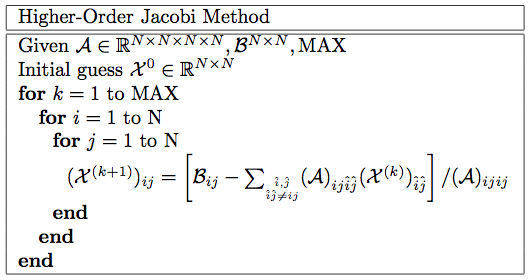}}
      \end{center}
  \caption{Psuedo-codes for Iterative Solvers.}
  \label{fig:algorithms}
\end{table}

\begin{figure}[htp]
  \begin{center}
    \subfigure[Aprroximated Solution]{\label{fig:poisson}\includegraphics[width = 70 mm]{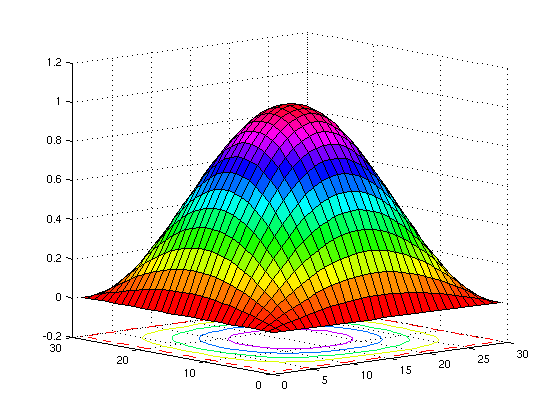}}
    \subfigure[Bicongugate Gradient (blue -.-) and Jacobi (red --)]{\label{fig:stiffJvBicg1}\includegraphics[width = 70 mm]{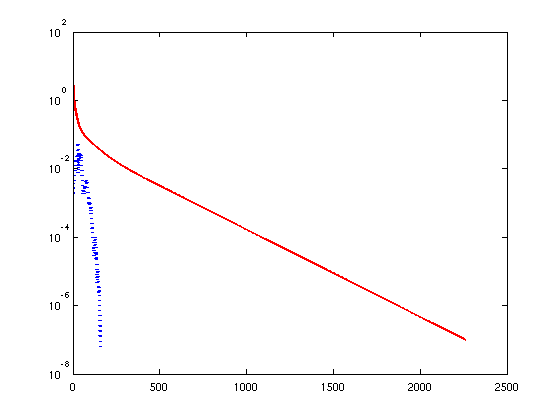}}
     \end{center}
  \caption{A solution to the Poisson equation in 2D with Dirichlet boundary conditions.}
  \label{fig:poissonplot}
\end{figure}

\section{An Eigenvalue Problem of the Anderson Model}
The Anderson model, Anderson's celebrated and ultimately Nobel prize winning work \cite{Anderson}, is the archetype and most studied model for understanding the spectral and transport properties of an electron in a disordered medium. In 1958, Anderson \cite{Anderson} described the behavior of electrons in a crystal with impurities, that is, when electrons can deviate from their sites by hopping from atom to atom and are constrained to an external random potential modeling the random environment.  This is called the tight binding approximation.  He argued heuristically that electrons in such systems result in a loss of the conductivity properties of the crystal, transforming it from conductors to insulators.

\subsection{The Anderson Model and Localization Properties}

The Anderson Model is a discrete random Schr\"odinger operator defined on a lattice $\mathbb{Z}^d$. More specifically, the Anderson Model is a random Hamiltonian $H_{\omega}$ on $\ell^2(\mathbb{Z}^d)$, $d \geq 1$, defined by
\begin{eqnarray}\label{Hamiltonian}
H_{\omega} = -\Delta + \lambda V_{\omega}
\end{eqnarray}
where $\Delta(x,y)=1$ if $\vert x-y  \vert =1$ and zero otherwise (the discrete Laplacian) with spectrum $[-2d,2d]$ and the random potential
$V_{\omega}=\{ V_{\omega}(x), x \in \mathbb{Z}^d\}$ consists of independent identically distributed random variables on $[-1,1]$ which we assume to have bounded and compactly supported density $\rho$. The disorder parameter  is the nonnegative $\lambda > 0$. The spectrum of $H_{\omega}$ can be explicitly described by
$$
\sigma(H_{\omega})= \sigma{(-\Delta)} + \lambda \mbox{~supp}(\rho) = [-2d,2d] + \lambda \mbox{~supp}(\rho).
$$
\begin{Remark}
The random potential $V_{\omega}$ is a multiplication operator on $\ell_2(\mathbb{Z}^d)$ with matrix elements $V_{\omega}(x)=v_x(\omega)$ where $(v_x(\omega))_{x \in \mathbb{Z}^d}$ is a collection of (i.i.d.) random variables with distribution $\rho$ indexed by $\mathbb{Z}^d$.
\end{Remark}
The random Schr\"odinger operator model disordered solids. The atoms or nuclei of a crystal are distributed in a lattice in a \emph{regular} way.  Since most solids are not ideal crystals, the positions of the atoms may deviate away from the ideal lattice positions.  This phenomena can be attributed to imperfections in the crystallization, glassy materials or a mixture of alloys or doped semiconductors. Thus to model disorder, a random potential $V_{\omega}$ perturbs the pure laplacian Hamiltonian ($-\Delta$) of a perfect metal. The time evolution of a quantum particle $\psi$ is determined by the Hamiltonian $H_{\omega}$; i.e.
\[ \psi(t)= e^{itH_{\omega}}\psi_0.\]
Thus the spectral properties of $H_{\omega}$ is studied to extract valuable information. The localization properties of the Anderson Model are of interest. 
For instance, the localization properties are characterized by the spectral properties of the Hamiltonian $H_{\omega}$; see the references \cite{Dirk,Kirsch,Stolz}. The Hamiltonian $H_{\omega}$ exhibits spectral localization if $H_{\omega}$ has almost surely pure point spectrum with exponentially decaying eigenfunctions. 

\begin{Remark}
Recall from \cite{ReedSimon} for any self-adjoint operator $H$, the spectral decomposition is
$$
\sigma(H)=\sigma_p(H) \cup \sigma_{ac}(H) \cup \sigma_{sc}(H)
$$
corresponding to the invariant subspaces $H_p$ of point spectrum, $H_{ac}$ of absolutely continuous and $H_{sc}$ to singular continuous spectrum.
\end{Remark}

The localization properties of the Anderson model can be described by spectral or dynamical properties. Let $I \subset \mathbb{R}$. 
\begin{definition}
We say that $H_{\omega}$ exhibits spectral localization in $I$ if  $H_{\omega}$ almost surely has
pure point spectrum in $I$ (with probability one), that is,
$$
\sigma(H_{\omega}) \cap I \subset \sigma_p (H_{\omega})~~\mbox{with~probability~one}
$$
Moreover, the random Schr\"odinger operator $H_{\omega}$ has exponential spectral localization in $I$ and the eigenfunctions corresponding to eigenvalues in $I$ decay exponentially.
\end{definition}
Thus if for almost all $\omega$, the random Hamiltonian $H_{\omega}$ has a complete set of eigenvectors $(\psi_{\omega,n})_{n \in \mathbb{N}}$ in the energy interval $I$ satisfying
$$
\vert  \psi_{\omega,n}(x) \vert  \leq  C_{\omega,n} e^{-\mu \vert  x - x_{\omega, n}  \vert}
$$
with localization center $x_{\omega, n}$ for $\mu > 0$ and $C_{\omega,n} < \infty$, then the exponential spectral localization hold on $I$.
\begin{Remark}
Let $V : \ell_2(\mathbb{Z}) \rightarrow \ell_2(\mathbb{Z})$ be a multiplication operator and suppose $v: \mathbb{Z} \rightarrow \mathbb{R}$ is a function. Then, $Vf(x)=v(x)f(x)$ and thus, $\sigma(V)=\mbox{range(v)}$. Suppose $f(x)$ is the Dirac delta function; i.e.
\begin{eqnarray*}
f(x)=\delta(x-x_0)=\begin{cases}
1 & x=x_0\\
0 & x\neq x_0,
\end{cases}
\end{eqnarray*}
\end{Remark}
then $Vf(x)=v(x_0)f(x)$ which implies that  $\sigma(V)=\sigma_p(V)$; i.e. $V$ has a pure point spectrum.

\begin{figure}[htp]
\begin{center}
 \subfigure[$\lambda=1,~N=50$]{\label{fig:Anderson1Da}\includegraphics[width = 75 mm]{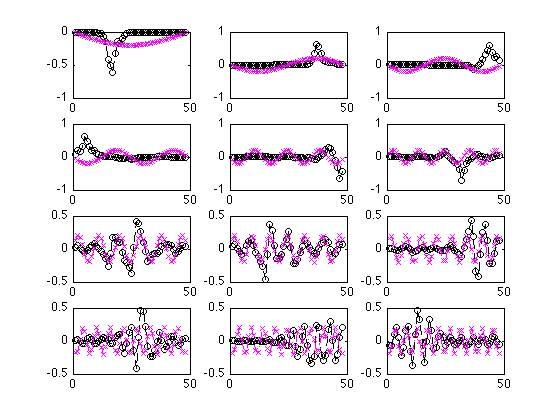}}
    \subfigure[$\lambda=.1,~N=50$]{\label{fig:Anderson1Db}\includegraphics[width = 75 mm]{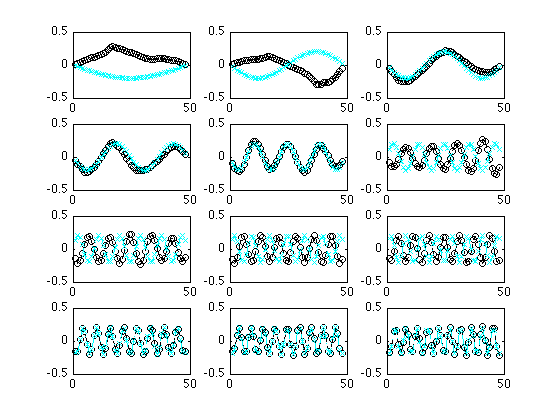}}
\end{center}
  \caption{One-dimensional Eigenvectors of the Discrete Schr\"odinger Operator (-x-) and the Anderson Model (black, -o-) for various modes.}
  \label{fig:Anderson1D}
\end{figure}

\begin{figure}[htp]
\begin{center}
 \subfigure[$\lambda=.1,~N=100$]{\includegraphics[width = 100 mm]{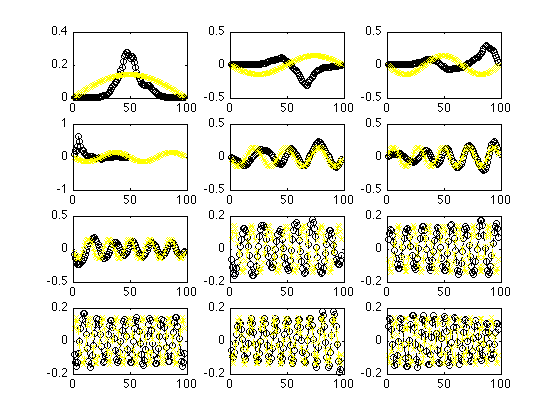}}
\end{center}
\caption{One-dimensional Eigenvectors of the Discrete Schr\"odinger Operator (-x-) and the Anderson Model (black, -o-) for various modes.}
\label{fig:Anderson1D2}
\end{figure}

\begin{definition}
A random Schr\"odinger operator has strong dynamical localization in an interval $I$ if for all $q>0$ and 
all $\phi \in \ell_2(\mathbb{Z}^d)$ with compact support
$$
\mathbb{E}\left [  \sup_t \Vert~\vert X \vert^q e^{-itH_{\omega}} \chi_I(H_{\omega}) \psi \Vert^2  < \infty \right ]
$$
where $\chi_I$ is an indicator function and $X$ is a multiplicative operator from $\ell_2(\mathbb{Z}^d) \rightarrow \ell_2(\mathbb{Z}^d)$ defined as $\vert X \vert \psi =\vert x \vert \psi(x)$.
\end{definition}
Dynamical localization in this form implies that all moments of the position operator are bounded in time.

As noted before, the Anderson model is a well-studied subject area for understanding the spectral and transport properties of an electron in a disordered medium, thus there are numerous results in both physics and mathematics literature; see \cite{Dirk} and the references therein. Mathematically, localization has been proven for the one-dimensional case for all energies and arbitrary disorder $\lambda$. For example, Kunz and Souillard \cite{KuSou} have proven in 1980 for $d=1$ and \emph{nice} distribution $\rho$ that localization is always present for any small disorder $\lambda$. In 1987, Carmona et al \cite{Carmona}  generalized this result in $d=1$ for any distribution $\rho$. In any $d$ dimension, for all energies and sufficiently large disorder ($\lambda >>1$), localized states are present.  For $d=2$ and for Gaussian distribution, it is conjectured that there is no extended state for any amount disorder $\lambda$ similar to the results for $d=1$. For $d \geq 3$, there exists $\lambda_0 > 0$ such that for $\lambda < \lambda_0$, $H$ has pure absolutely continuous spectrum. It is known (see \cite{Dirk}) that there exist $\lambda_1 < \infty$ such that for $\lambda > \lambda_1$, $H_{\lambda}$ has dense pure spectrum. There are still many open problems like the extended state conjecture \cite{Erdos}.

\subsection{Approximation of Eigenvectors}

To approximate the eigenvectors of the multidimensional Anderson model, the eigenvalue decomposition in Theorem \ref{TEVD} is applied to the Hamiltonian $H_{\omega}$. The Hamiltonian $H_{\omega}$ in two and three dimensions are formed into fourth- and sixth-order tensors using the same stencils in Figure \ref{fig:plot2} corresponding to the entries in (\ref{fivepoint}) and (\ref{sevenpoint}), respectively. The only main differences are that the center nodes are centered around zero and have random entries, 
\begin{eqnarray}
({H_{\omega}}^{(3,4)}_{k=\alpha,l=\beta})_{\alpha,\beta}=\frac{\sigma}{(\Delta x)^2} \label{hamiltontensor1}\\
\mbox{and} \nonumber\\
({H_{\omega}}^{(4,5,6)}_{l=\alpha,m=\beta,n=\gamma})_{\alpha,\beta,\gamma}=\frac{\tau}{(\Delta x)^3} \label{hamiltontensor2}
\end{eqnarray}
where $\sigma$ and $\tau$ are random numbers with uniform distribution on $[-1,1]$ accounting for the random diagonal potential $V_{\omega}$. With the formulations of the Hamiltonian like (\ref{kron1}) and (\ref{kron2}), the uniform distribution on $[-1,1]$ on the random potential cannot be guaranteed. But the higher-order tensor representation easily preserved this structure. To numerically compute the higher-dimensional eigenvector, tensor representation of the Hamiltonian is necessary before the appropriate Einstein product rules and mappings are applied.

\begin{figure}[htp]
\begin{center}
 \subfigure[N=29]{\label{fig:Anderson2Da}\includegraphics[width = 133 mm]{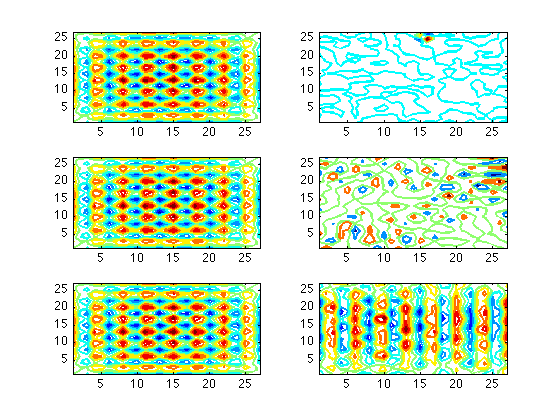}}
    \subfigure[N=48]{\label{fig:Anderson2Db}\includegraphics[width = 133 mm]{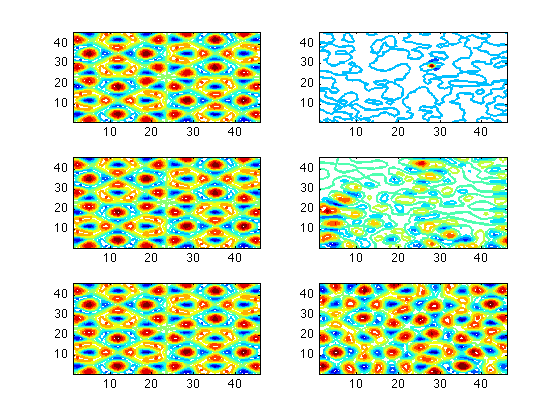}}
\end{center}
  \caption{Two-dimensional Eigenvectors of the Discrete Schr\"odinger Operator (left column) and the Anderson Model (right column) for varying disorder ($\lambda=10$ (top), $\lambda=1$ (middle) and $\lambda=.1$ (bottom)).}
  \label{fig:Anderson2D}
\end{figure}

\begin{figure}[htp]
\begin{center}
\includegraphics[width = 100 mm]{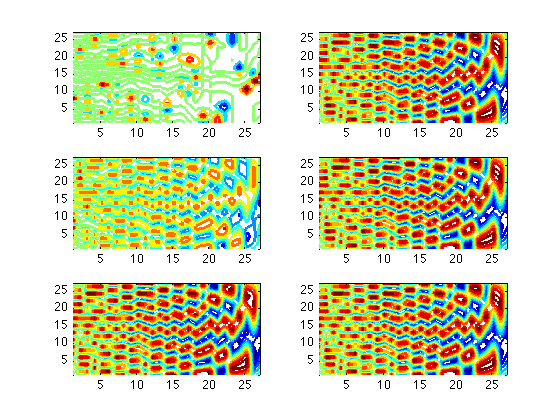}
\end{center}
\caption{Factors of the Multilinear SVD Decomposition \cite{HOSVD} of the Two-dimensional Discrete Schr\"odinger Operator (right column) and the Anderson Model (left column) for varying disorder ($\lambda=10$ (top), $\lambda=1$ (middle) and $\lambda=.1$ (bottom)).}
\label{fig:Anderson2Db}
\end{figure}

In Figures \ref{fig:Anderson1D}, \ref{fig:Anderson1D2}, \ref{fig:Anderson2D} ,\ref{fig:Anderson2Db} and \ref{fig:Anderson3D},  the eigenfunctions are approximated by the eigenvectors from the both discrete Schr\"odinger and random Schr\"odinger (Anderson) models.  In Figures \ref{fig:Anderson1D} and \ref{fig:Anderson1D2}, the eigenvectors of the Anderson Model in one dimension are definitely more localized than the eigenvectors of the discrete random Schr\"odinger model in one dimension which are consistent with the results in \cite{Dirk} for the Anderson model in one dimension. Observe that for large amount of disorder (e.g. $\lambda=1$), the localized states are apparent.  However this is not true for smaller amount of disorder (e.g. $\lambda=.1$). The localization is not so apparent for  $\lambda=.1$ for $N=50$, but when the number of atoms is increased, that is, setting $N=100$, the localized eigenvectors are present as in the case when $\lambda=1$; see Figures \ref{fig:Anderson1D} (part b) and \ref{fig:Anderson1D2}

In the contour plots of Figures \ref{fig:Anderson2D}, \ref{fig:Anderson2Db} and \ref{fig:Anderson3D}, the eigenvectors in two and three dimensions of the Anderson model are more \emph{peaked} than those of the nonrandomized Schr\"odinger for large disorder $\lambda \geq 1$. As in the case for one dimension, localization is not apparent for small disorder ($\lambda=.1$) as seen in Figure \ref{fig:Anderson2D}. Moreover, as $N$ increases, for small disorder the eigenstates of both discrete Schr\"odinger and Anderson models seems to coincide. This does not necessarily mean that the localization is absent for this regime, but rather the localized states are harder to find for small amount of disorder and a larger amount of atoms have to be considered. In Figure \ref{fig:Anderson2Db}, localization is not clearly visible for even $\lambda=1$ in the factors calculated via the multilinear SVD decomposition \cite{HOSVD} while localization is detected in the plots in Figure \ref{fig:Anderson2D} when $\lambda=1$. The plots in Figure \ref{fig:Anderson2Db} are generated by applying the HOOI algorithm \cite{HOOI} to the Hamiltonian tensors (\ref{hamiltontensor1},\ref{hamiltontensor2}). 

Our numerical results provide some validation that these localizations exist for large disorder for dimension $d>1$ for sufficient amount of atoms. 

\begin{figure}[htp]
  \begin{center}
    \subfigure[$\lambda=10$]{\label{fig:Anderson3Da}\includegraphics[width = 52 mm]{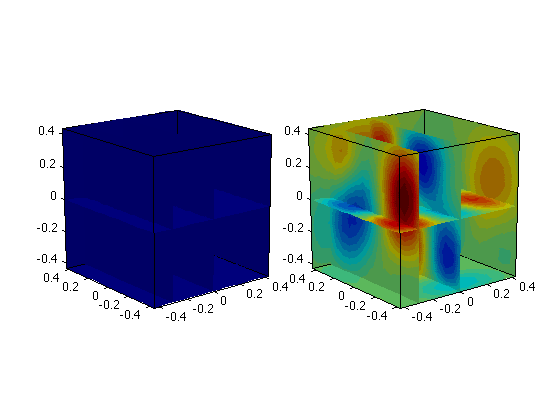}}
    \subfigure[$\lambda=1$]{\label{fig:Anderson3Da2}\includegraphics[width = 52 mm]{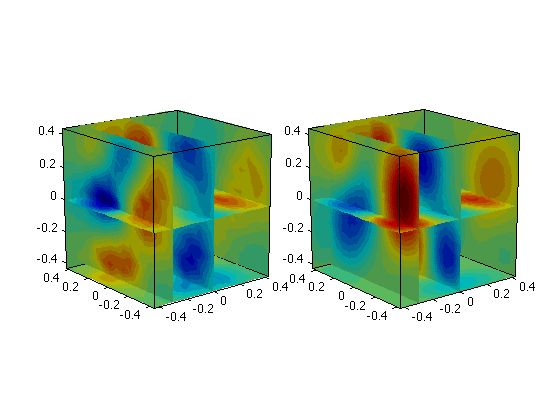}} 
    \subfigure[$\lambda=0.1$]{\label{fig:Anderson3Da3}\includegraphics[width = 52 mm]{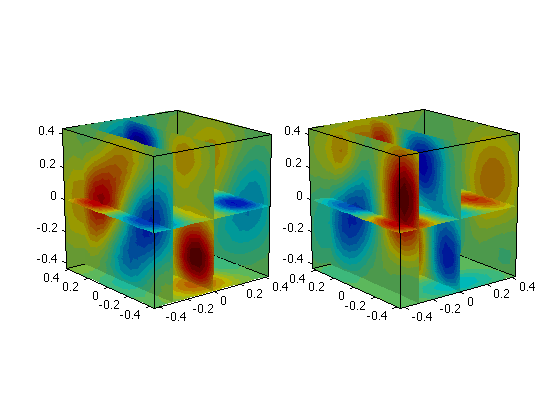}}
    \subfigure[$\lambda=10$]{\label{fig:Anderson3Db}\includegraphics[width = 52 mm]{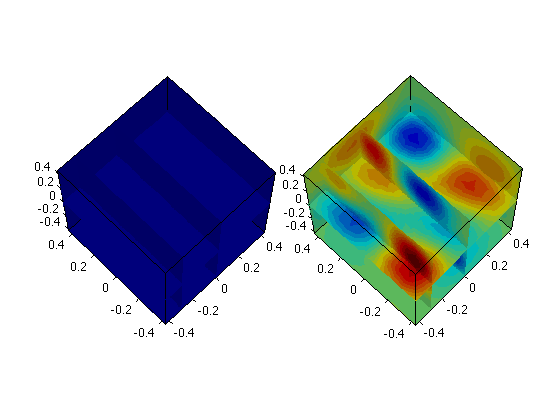}}
    \subfigure[$\lambda=1$]{\label{fig:Anderson3Db2}\includegraphics[width = 52 mm]{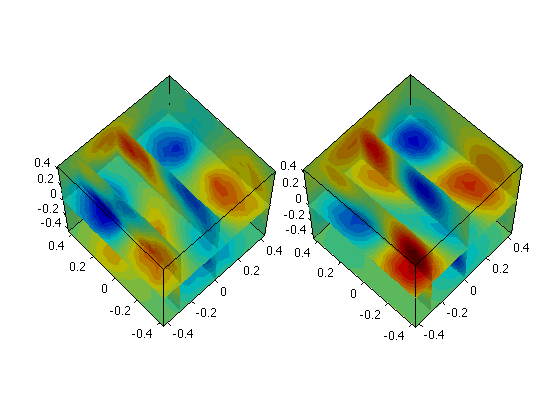}} 
    \subfigure[$\lambda=0.1$]{\label{fig:Anderson3Db3}\includegraphics[width = 52 mm]{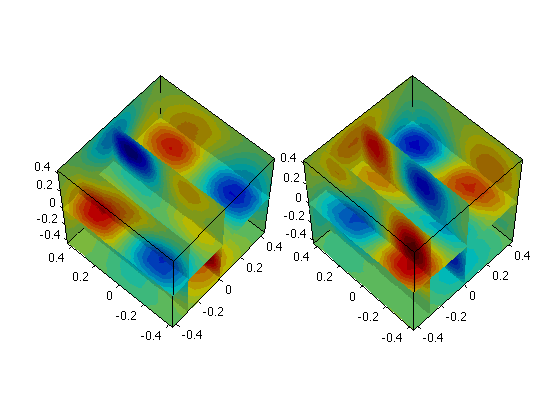}}
  \end{center}
  \caption{Two views (front (first row) and top (second row)) of the Three-dimensional Eigenvectors of the Anderson Model (left) and the Discrete Schr\"odinger Operator (right) for varying disorder.}
  \label{fig:Anderson3D}
\end{figure}

\section{Multilinear Least Squares}

Under the Einstein product rule, odd-order and nonhyper-rectangular tensors do not have inverses. In this section, we extend the concepts of pseudo-inversion for odd-order tensors and and nonhyper-rectangular tensors.

\subsection{Least-Squares}
The linear least-squares (LLS) method is a well-known method for data analysis.  Often the number of observations $\bold{b}$ exceed the number of unknown parameters $\bold{x}$ in LLS, forming an overdetermined system, e.g. 
\begin{eqnarray}\label{overdeter} 
\bold{A} \bold{x} =\bold{b} 
\end{eqnarray}
where $\bold{A} \in \mathbb{R}^{m \times n}$, $\bold{x} \in \mathbb{R}^n$, and $\bold{b} \in \mathbb{R}^m$ with $m > n$.  Through minimization of the residual, $\bold{r}=\bold{b}-\bold{Ax}$, the overdetermined system (\ref{overdeter}) can be solved. If the objective function being minimized over $\bold{x} \in \mathbb{R}^n$ is $\phi(\bold{x})=\Vert \bold{r}\Vert_{\ell_2}$, then this is the least-squares method. Thus, the solution obtained through LLS is the vector $\bold{x^*} \in \mathbb{R}^n$ minimizing $\phi(\bold{x})$; the vector $\bold{x}^*$ is called the \emph{least-squares} solution of the linear system (\ref{overdeter}).

Here are examples of overdetermined multilinear systems. 
\begin{itemize}
\item[(i)]
$\mathcal{A} \bullet_3 \bold{x} = \bold{B}$ where $\mathcal{A} \in \mathbb{R}^{I \times J \times K}$, $\bold{x} \in \mathbb{R}^K$ and $\bold{B} \in \mathbb{R}^{I \times J}$\label{mls1}
\item[(ii)]
$\mathcal{A} \ast \mathcal{X} = \mathcal{B}$ where  $\mathcal{A} \in \mathbb{R}^{I \times J \times R \times S}, \mathcal{X} \in \mathbb{R}^{R \times S \times K \times L}$ and $\mathcal{B} \in \mathbb{R}^{I \times J \times K \times L}$ 
\end{itemize}
For both cases, higher-order tensor inverses of $\mathcal{A}$ do not exist. The formulations,
\begin{eqnarray}
\min_{\bold{x}} \Vert  \mathcal{A} \bullet_3 \bold{x} - \bold{B} \Vert _F \mbox{~~~and~~~} \min_{\mathcal{X}} \Vert  \mathcal{A} \ast \mathcal{X} - \mathcal{B} \Vert_F,
\end{eqnarray}
are considered to find \emph{multilinear least-squares} solutions of systems. Note that the Frobenius norm, $\Vert \cdot \Vert_F$, is defined as $\Vert \mathcal{A} \Vert_F^2=\sum_{i_1i_2\ldots i_N} \vert  a_{i_1i_2\ldots,i_N} \vert^2$ for $\mathcal{A}^{I_1\times I_2 \times \ldots I_N}$.

\subsection{Normal equations}

\begin{definition}[Critical Point]
Let $\phi:\mathbb{R}^n \rightarrow \mathbb{R}$ be a continuously differentiable function. A critical point of $\phi$ is a point $\bar{\bold{x}} \in \mathbb{R}^n$ such that 
\[ \nabla \phi(\bar{\bold{x}})=\bold{0}.   \]
\end{definition}
Consider the multilinear system,
\begin{eqnarray} \label{mls}
\mathcal{A} \bullet_3 \bold{x} = \bold{B} 
\end{eqnarray}
where $\mathcal{A} \in \mathbb{R}^{I \times J \times K}$, $\bold{x} \in \mathbb{R}^K$, and $\bold{B} \in \mathbb{R}^{I \times J}$ and define
\begin{eqnarray}\label{objfunc1}
\phi_1(\bold{x})=\Vert  \mathcal{A} \bullet_3 \bold{x} - \bold{B} \Vert _F^2.
\end{eqnarray}

\begin{lemma} \label{lemmamls}
Any minimizer $\bar{\bold{x}} \in \mathbb{R}^K$ of $\phi_1$ satisfies the following system
\begin{eqnarray}\label{normaleqn1}
\mathcal{A}^T \ast_2 \mathcal{A} \bullet_3 \bold{x} = \mathcal{A}^T  \ast_2 \bold{B}.
\end{eqnarray}
\end{lemma}
\begin{proof}
We expand the objective function, 
\begin{eqnarray*}
\phi_1(\bold{x})= \langle \mathcal{A} \bullet_3 \bold{x} - \bold{B}, \mathcal{A} \bullet_3 \bold{x} - \bold{B}   \rangle &=& \langle \mathcal{A} \bullet_3 \bold{x} ,\mathcal{A} \bullet_3 \bold{x}   \rangle -2 \langle \mathcal{A} \bullet_3 \bold{x}  , \bold{B}  \rangle 
+  \langle \bold{B},\bold{B} \rangle \\
&=& (\mathcal{A} \bullet_3 \bold{x} )^T (\mathcal{A} \bullet_3 \bold{x}) -2 \bold{B}^T( \mathcal{A} \bullet_3 \bold{x}) + \bold{B}^T \bold{B}.
\end{eqnarray*}
Then,
\begin{eqnarray}\label{normeq1}
\nabla \phi_1(\bold{x})=\frac{\partial}{\partial \bold{x}} \left [ (\mathcal{A} \bullet_3 \bold{x} )^T (\mathcal{A} \bullet_3 \bold{x}) \right ]&=&\frac{\partial}{\partial \bold{x}} \left [  \sum_{ij} \left ( \sum_{kl}   \bold{x}_k \mathcal{A}_{kij}\mathcal{A}_{ijl} \bold{x}_l    \right ) \right ]
=\frac{\partial}{\partial \bold{x}} \left [  \sum_{kl} \left ( \sum_{ij}   \bold{x}_k \mathcal{A}_{kij}\mathcal{A}_{ijl} \bold{x}_l    \right ) \right ]  \nonumber\\
&=&\frac{\partial}{\partial \bold{x}} \left [  \sum_{kl} \bold{x}_k (\mathcal{A}^T \ast \mathcal{A})_{kl} \bold{x}_l   \right ]
= 2 (\mathcal{A}^T \ast_2 \mathcal{A}) \bold{x}  \nonumber\\
&=&  2 \mathcal{A}^T \ast_2  \mathcal{A} \bullet_3 \bold{x}
\end{eqnarray}
and
\begin{eqnarray}\label{normeq2}
 2 \frac{\partial}{\partial \bold{x}} \left [ ( \mathcal{A} \bullet_3 \bold{x})^T\bold{B} \right ] &=&2\frac{\partial}{\partial \bold{x}} \left [  \sum_{ij} \left ( \sum_{k}  \bold{x}_k \mathcal{A}_{kij}  \bold{B}_{ij}    \right ) \right ] =2 \frac{\partial}{\partial \bold{x}} \left [  \sum_{k} \left ( \sum_{ij}    \bold{x}_k \mathcal{A}_{kij}\bold{B}_{ij}  \right ) \right ] \nonumber \\
 &=&2 \frac{\partial}{\partial \bold{x}} \left [  \sum_{k}   \bold{x}_k (\mathcal{A}^T \ast_2 \bold{B})_k   \right ]  \nonumber\\
 &=& 2 \mathcal{A}^T \ast_2 \bold{B}
  \end{eqnarray}
where $\mathcal{A}^T$, a permutation of $\mathcal{A}$ where $(\mathcal{A}^T)_{kij}=(\mathcal{A})_{ijk}$. Thus from (\ref{normeq1}-\ref{normeq2}),
\[ \frac{\partial \phi_1}{\partial \bold{x}}(\bold{x})= 2 \mathcal{A}^T \ast_2  \mathcal{A} \bullet_3 \bold{x}-2 \mathcal{A}^T \ast_2 \bold{B}.\] 
Clearly, the minimizer  $\bar{\bold{x}}$ of $\phi_1$ satisfies
\begin{eqnarray*}\label{normeq}
\mathcal{A}^T \ast_2  \mathcal{A} \bullet_3 \bold{x}= \mathcal{A}^T \ast_2 \bold{B}.
\end{eqnarray*}
Furthermore, the critical point is $\bar{\bold{x}}= (\mathcal{A}^T \ast_2  \mathcal{A})^{-1} \ast_2 \mathcal{A}^T \ast_2 \bold{B}$.
\end{proof}
For the problem
\[\mathcal{A} \ast_2 \mathcal{X} = \mathcal{B}\]
where  $\mathcal{A} \in \mathbb{R}^{I \times J \times R \times S}, \mathcal{X} \in \mathbb{R}^{R \times S \times K \times L}$ and $\mathcal{B} \in \mathbb{R}^{I \times J \times K \times L}$ and the objective function,
\begin{eqnarray}\label{objfunc2}
\phi_2(\bold{x})=\Vert  \mathcal{A} \ast_2 \mathcal{X} - \mathcal{B} \Vert _F^2,
\end{eqnarray}
we have the following lemma.
\begin{lemma} \label{lemmamls2}
Any minimizer $\bar{\mathcal{X}} \in \mathbb{R}^{R \times S \times K \times L}$ of $\phi_2$ satisfies the following system
\begin{eqnarray}\label{normaleqn2}
\mathcal{A}^T \ast_2 \mathcal{A} \ast_2 \mathcal{X}= \mathcal{A}^T \ast_2 \mathcal{B}
\end{eqnarray}
where $\mathcal{A}^T \in \mathbb{R}^{R \times S \times I \times J}$ denotes the transpose of $\mathcal{A} \in \mathbb{R}^{I \times J \times R \times S}$. Moreover, the critical point of $\phi_2$ is $\bar{\mathcal{X}}=(\mathcal{A}^T \ast_2 \mathcal{A})^{-1} \ast_2 \mathcal{A}^T \ast_2 \mathcal{B}$.
\end{lemma}

\begin{remark}
We omit the proof for Lemma \ref{lemmamls2} since it is similar to that of Lemma \ref{lemmamls}.  Both critical points,  $\bar{\bold{x}}= (\mathcal{A}^T \ast_2  \mathcal{A})^{-1} \ast_2 \mathcal{A}^T \ast_2 \bold{B}$ and $\bar{\mathcal{X}}=(\mathcal{A}^T \ast_2 \mathcal{A})^{-1} \ast_2 \mathcal{A}^T \ast_2 \mathcal{B}$ are unique minimizers for (\ref{objfunc1}) and (\ref{objfunc2}), respectively, since $\phi_1$ and $\phi_2$ are quadratic functions. Equations (\ref{normaleqn1}) and (\ref{normaleqn2}) are called the high-order normal equations.
\end{remark}

\subsection{Transposes and permutations}

From Definition \ref{tensortranspose}, the transpose of $\mathcal{A} \in \mathbb{R}^{I \times J \times R \times S}$ in (\ref{normaleqn2}) is easily obtained. Since the definition only holds for even-order tensors, we extend the notion of transposition to third (odd) order tensors.  Recall in Lemma \ref{lemmamls}, we have denoted a permutation of $\mathcal{A} \in \mathbb{R}^{I \times J \times K}$ as $\mathcal{A}^T \in \mathbb{R}^{K \times I \times J}$. The transpose of a third order tensor $\mathcal{A} \in \mathbb{R}^{I \times J \times K}$ is a permutation since third order tensors can be viewed as fourth order tensors with one mode in one-dimension.  For example, if $\mathcal{B}$ is a permutation of $\mathcal{A} \in \mathbb{R}^{I \times J \times K \times L}$ with $L=1$ , then $b_{klij}=a_{ijkl}$ which is $b_{ijk1}=a_{k1ij} \Leftrightarrow b_{ijk}=a_{kij}.$ Thus we denote $\mathcal{B}=\mathcal{A}^T$ where $\mathcal{B} \in \mathbb{R}^{K \times I \times J}$.

Unlike in the matrix case where $(\bold{A}^T)^T=\bold{A}$, for third order tensors we have the following property.  
\begin{lemma}[Property of third order tensor transpose]
Let $\mathcal{A} \in \mathbb{R}^{I \times J \times K}$ and $\rho$ be a permutation on the index set $\{ijk\}$. Then $((\mathcal{A}^T)^T)^T=\mathcal{A}$.
\end{lemma}
\begin{proof}
For the index set $\{ijk\}$, there are two cyclic permutations: $\rho_1(ijk)=jki,~  \rho_2(jki)=kij,~  \rho_3(kij)=ijk$ and $\bar{\rho_1}(ikj)=kji,~  \bar{\rho_2}(kji)=jik,~  \bar{\rho_3}(jik)=ikj$. It follows that $(((\mathcal{A}^T)^T)^T)_{ijk}=(\mathcal{D})_{ijk}=(\mathcal{D})_{\rho_3(kij)} \Longrightarrow 
((\mathcal{A}^T)^T)_{kij}=(\mathcal{D})_{kij}=(\mathcal{D})_{\rho_2(jki)}  \Longrightarrow 
(\mathcal{A}^T)_{jki}=(\mathcal{D})_{jki}=(\mathcal{D})_{\rho_1(ijk)} \Longrightarrow 
(\mathcal{A})_{ijk}=(\mathcal{D})_{ijk}$.
\end{proof}
There are six permutations for a third order tensor, although there are two cyclic permutations.  For an $N$th order tensor, the number of tensor transposes is dependent on the number and length of cyclic permutations on the index set $\{i_1 i_2 \ldots i_N\}$. Table $2$ lists all the possible multilinear least squares problems for third order tensors and their corresponding tensor transposes.
\begin{table}[h]
\begin{center}
\begin{tabular}{|c|c|c|c|}
\hline 
$\mathcal{A}$ & $\bold{x}$ & $\bold{B}$ &  $\mathcal{A}^T$\\
\hline
$\mathbb{R}^{I \times J \times K}$ & $\mathbb{R}^K$ & $\mathbb{R}^{I \times J}$ & $\mathbb{R}^{K \times I \times J}$\\
\hline
$\mathbb{R}^{J \times K \times I}$ & $\mathbb{R}^I$ & $\mathbb{R}^{J \times K}$ & $\mathbb{R}^{I \times J \times K}$\\
\hline
$\mathbb{R}^{K \times I \times J}$ & $\mathbb{R}^J$ & $\mathbb{R}^{K \times I}$ & $\mathbb{R}^{J \times K \times I}$\\
\hline \hline
$\mathbb{R}^{I \times K \times J}$ & $\mathbb{R}^J$ & $\mathbb{R}^{K \times I}$ & $\mathbb{R}^{J \times I \times K}$\\
\hline
$\mathbb{R}^{K \times J \times I}$ & $\mathbb{R}^I$ & $\mathbb{R}^{K \times J}$ & $\mathbb{R}^{I \times K \times J}$\\
\hline
$\mathbb{R}^{J \times I \times K}$ & $\mathbb{R}^K$ & $\mathbb{R}^{J \times I}$ & $\mathbb{R}^{K \times J \times I}$\\
\hline
\end{tabular}
	\caption{Dimensions for Higher-Order Normal Equations for Third Order Tensors}
\end{center}
\end{table}


\subsection*{Acknowledgments}
C.N. and N.L. are both in part supported by National Science Foundation DMS-0915100. C.N. would like to thank Shannon Starr for some fruitful discussions on the quantum models.

\subsection*{Appendix: Proof of Theorem \ref{c'estgroup}}

\begin{proof}
Here we prove the main theorem by checking each axioms ($A1-A3$) hold in Definition $3.3$. 
\begin{itemize}
\item[$(A1)$]
Show that the binary operation $\ast_2$ is associative.  

Since we know that $f$ is a bijective map with the property that $f(\mathcal{A} \ast_2 \mathcal{B})= f(\mathcal{A}) \cdot f(\mathcal{B})$. We will show $f^{-1}(\bold{A} \cdot \bold{B})= f^{-1}(\bold{A}) \ast_2 f^{-1}(\bold{B})$, for $\bold{A}, \bold{B} \in \mathbb{M}$.

Let $\mathcal{A}, \mathcal{B}, \mathcal{C} \in \mathbb{T}$ and $\bold{A}, \bold{B}, \bold{C} \in \mathbb{M}$ where $f(\bold{A})=\mathcal{A}$,  $f(\bold{B})=\mathcal{B}$ and $f(\bold{C})=\mathcal{C}$. Then,
\begin{eqnarray*}
(\mathcal{A}\ast_2 \mathcal{B}) \ast_2 \mathcal{C}
& = & f^{-1}(\bold{A}) \ast_2 f^{-1}(\bold{B}) \ast_2 f^{-1}(\bold{C})=f^{-1}(\bold{A} \cdot \bold{B} \cdot \bold{C})=f^{-1}(\bold{A} \cdot (\bold{B} \cdot \bold{C})) \\
 & = & f^{-1}(\bold{A}) \ast_2 f^{-1}(\bold{B} \cdot \bold{C})=\mathcal{A} \ast_2(f^{-1}(\bold{B}) \ast_2 f^{-1}(\bold{C}))=\mathcal{A} \ast_2 (\mathcal{B} \ast_2 \mathcal{C})
\end{eqnarray*}

Therefore, $(\mathcal{A}\ast_2 \mathcal{B}) \ast_2 \mathcal{C} = \mathcal{A} \ast_2 (\mathcal{B} \ast_2 \mathcal{C})$.

\item[$(A2)$]
Show that there is an identity element for $\ast_2$ on $\mathbb{T}$. 

Since $\bold{I}^{I_{1}I_{2} \times I_{1}I_{2}} \in \mathbb{M}$ is the identity element in the group. Note that we will suppress the superscript of $\bold{I}$ in the calculation below. Then we claim that $f^{-1}({\bold{I}})$ is the identity element for $\ast_2$ on $\mathbb{T}$.

For every element $\mathcal{A} \in \mathbb{T}$, there exists a matrix $A \in \mathbb{M}$ so that $f^{-1}(\bold{A})=\mathcal{A}$. So, we get
$$\mathcal{A} \ast_2 f^{-1}(\bold{I})=f^{-1}(\bold{A}) \ast_2 f^{-1}(\bold{I})=f^{-1}(\bold{A} \cdot \bold{I})=f^{-1}(\bold{A})=\mathcal{A}$$
Similarly,
$$f^{-1}(\bold{I}) \ast_2 \mathcal{A}=f^{-1}(\bold{I}) \ast_2 f^{-1}(\bold{A})=f^{-1}(\bold{I} \cdot \bold{A})=f^{-1}(\bold{A})=\mathcal{A}$$
Therefore, $\mathcal{A} \ast_2 f^{-1}(\bold{I})=f^{-1}(\bold{I}) \ast_2 \mathcal{A} =\mathcal{A}$.

Define the tensor $\mathscr{E}$ as follows
\begin{eqnarray*}
(\mathscr{E})_{i_1 i_2 j_1 j_2}=\delta_{i_1 j_1}\delta_{i_2 j_2}
\end{eqnarray*}
where
\begin{eqnarray*}
\delta_{lk}=
\begin{cases}
$1,$ & \mbox{$l = k$} \\
$0,$ & \mbox{$l \neq k$}
\end{cases}
\end{eqnarray*}
We claim that $\mathscr{E}=f^{-1}(\bold{I})$. By direct calculations, we have 
$$(\mathscr{E} \ast_2 \mathcal{A})_{i_1 i_2 j_1 j_2} = \sum_{u,v} \epsilon_{i_1 i_2 u v}a_{u v j_1 j_2} = \epsilon_{i_1 i_2 i_1 i_2}a_{i_1 i_2 j_1 j_2}= \delta_{i_1 i_1}\delta_{i_2 i_2}a_{i_1 i_2 j_1 j_2}= a_{i_1 i_2 j_1 j_2}= \mathcal{A}_{i_1 i_2 j_1 j_2}$$
and
$$(\mathcal{A} \ast_2 \mathscr{E})_{i_1 i_2 j_1 j_2}= \sum_{u,v}a_{i_1 i_2 u v}\epsilon_{u v j_1 j_2} = a_{i_1 i_2 j_1 j_2}\epsilon_{j_1 i_2 j_1 j_2}=a_{i_1 i_2 j_1 j_2} \delta_{j_1 j_1}\delta_{j_2 j_2}= a_{i_1 i_2 j_1 j_2}= \mathcal{A}_{i_1 i_2 j_1 j_2}.$$

Thus $\mathscr{E} \ast_2 \mathcal{A} = \mathcal{A} \ast_2 \mathscr{E}=\mathcal{A}$, for $\forall \mathcal{A} \in \mathbb{T}$. Therefore $\mathscr{E}_{i_1 i_2 j_1 j_2}=\delta_{i_1 j_1}\delta_{i_2 j_2}$ is the identity element for $\ast_2$ on $\mathbb{T}$.

Finally, we know that $f^{-1}(\bold{I}^{I_{1}I_{2}\times I_{1}I_{2}})=\mathscr{E}$.

\item[$(A3)$]
Show that for each $\mathcal{A} \in \mathbb{T}$, there exists an inverse
$\widetilde{\mathcal{A}}$ such that $\widetilde{\mathcal{A}} \ast_2 \mathcal{A}= \mathcal{A} \ast_2 \widetilde{\mathcal{A}} = \mathscr{E}$.

 We define $\widetilde{\mathcal{A}}= f^{-1}\{[f(\mathcal{A})]^{-1}\}$ since $f(\mathcal{A}) \in \mathbb{M}$ and $f^{-1}$ is a bijection map from Lemma (\ref{lemmaf}). Then,
 $$f(\widetilde{\mathcal{A}} \ast_2 \mathcal{A})=f(\widetilde{\mathcal{A}}) \cdot f(\mathcal{A})=[f(\mathcal{A})]^{-1} \cdot f(\mathcal{A})=\bold{I}^{I_1 I_2 \times I_1 I_2}$$
 
 From Lemma \ref{lemmaf} and since $f(\mathscr{E})=\bold{I}^{I_1 I_2 \times I_1 I_2}$, we obtain $\widetilde{\mathcal{A}} \ast_2 \mathcal{A}=\mathscr{E}.$
 
 Similarly, we can get $ \mathcal{A} \ast_2 \widetilde{\mathcal{A}} = \mathscr{E}$. 
 
 It follows that  for each $\mathcal{A} \in \mathbb{T}$, there exists an inverse
$\widetilde{\mathcal{A}}$ such that $\widetilde{\mathcal{A}} \ast_2 \mathcal{A}= \mathcal{A} \ast_2 \widetilde{\mathcal{A}} = \mathscr{E}.$
 
 \end{itemize}
Therefore, the ordered pair $(\mathbb{T},\ast_2)$ is a group where the operation $\ast_2$ is defined in (\ref{eins1}). In addition, the transformation $f:\mathbb{T} \rightarrow \mathbb{M}$ (\ref{transf1}) is a bijective mapping between groups. Hence, $f$ is an isomorphism.
\end{proof}

\end{document}